\documentclass[leqno,11pt]{article}
\usepackage{amsmath,amssymb,amsthm,graphicx,color}
\usepackage{graphicx}\usepackage{epsfig,graphics}
\usepackage{epstopdf}

\newcommand{\bas}{\begin{eqnarray*}} \newcommand{\eas}{\end{eqnarray*}}
\def\tto{\;{\lower 1pt \hbox{$\rightarrow$}}\kern -12pt
           \hbox{\raise 2.8pt \hbox{$\rightarrow$}}\;}
\def\for{\hskip0.9pt|\hskip0.9pt}
\def\N{{\bf N}}

\usepackage[normalem]{ulem}

\newcommand{\bs}{{\bigskip}} \newcommand{\ms}{{\medskip}}
\newcommand{\bd}{\begin{displaymath}} \newcommand{\ed}{\end{displaymath}}
\newcommand{\be}{\begin{equation}} \newcommand{\ee}{\end{equation}}
\newcommand{\ba}{\begin{eqnarray}} \newcommand{\ea}{\end{eqnarray}}
\def\paritem#1{\vskip0cm\noindent\hskip12pt{{\rm #1}}\hskip5pt}
\def\bx{\bar x} \def\by{\bar y}  \def\bp{\bar p}
\def\ba{\bar a}  \def\bz{\bar z} \def\bu{\bar u}
\def\dom{\mathop{\rm dom}\nolimits} \def\rge{\mathop{\rm rge}\nolimits}
\def\gph{\mathop{\rm gph}\nolimits} 
 
\def\reg{\mathop{\rm reg}\nolimits} \def\clm{\mathop{\rm clm}\nolimits}

\def\half{{{}\raise 1pt \hbox{${{\scriptstyle1}\over{\scriptstyle 2}}$}}}
 \def\inT{\mathop{\rm int}\nolimits}
\def\oball{{\rm int}{I\kern -.35em B}}
\def\plus{{\scriptscriptstyle +}} \def\minus{{\scriptscriptstyle -}}
\def\starplus{{*\hskip-0.6pt{\raise0.5pt\hbox{$\plus$}}}}
\def\starminus{{*\hskip-0.6pt{\raise0.5pt\hbox{$\minus$}}}}

\def\epsilon{\varepsilon}             \def\phi{\varphi}
\def\reals{{I\kern-.35em R}}

\def\text#1{\;\,\hbox{#1}\;\,}       

\def\state #1. { \noindent{\bf#1.\enspace}}
\def\lset{\big\{\,}    \def\mset{\,\big|\,}   \def\rset{\big\}}
\def\Lset{\Big\{\,}    \def\Mset{\,\Big|\,}   \def\Rset{\Big\}}
\outer\def\proclaim #1. #2\par{\medbreak \noindent{\bf#1.\enspace}{\sl#2}\par
\ifdim\lastskip<\medskipamount \removelastskip\penalty55\medskip\fi}
\outer\def\proclaime #1. #2\par{\medbreak \noindent{\bf#1.\enspace}{\rm#2}\par
 \ifdim\lastskip<\medskipamount \removelastskip\penalty55\medskip\fi}
\def\qed{\hfill{$\vcenter{\hrule height1pt \hbox{\vrule width1pt height5pt
   \kern5pt \vrule width1pt} \hrule height1pt}$} \medskip}
\def\low#1{{\lower1pt \hbox{$\scriptstyle #1$}}}
\def\high#1{{\raise1pt \hbox{$\scriptstyle #1$}}}
\def\downto{{\raise 1pt \hbox{$\scriptstyle \,\searrow\,$}}}

\def\bx{\bar x} \def\by{\bar y} 

\def\dom{\mathop{\rm dom}\nolimits} \def\rge{\mathop{\rm rge}\nolimits}
\def\gph{\mathop{\rm gph}\nolimits} 

\def\la{\langle}\def\ra{\rangle}
\def\subreg{\mathop{\rm subreg }\nolimits}

\def\reg{\mathop{\rm reg}\nolimits} 
 
 \def\inT{\mathop{\rm int}\nolimits}
\def\plus{{\scriptscriptstyle +}} \def\minus{{\scriptscriptstyle -}}

\def\epsilon{\varepsilon}  \def\eps{\varepsilon}           \def\phi{\varphi}
\def\reals{{I\kern-.35em R}}

\def\text#1{\;\,\hbox{#1}\;\,}       

\def\state #1. { \noindent{\bf#1.\enspace}}
\def\lset{\big\{\,}    \def\mset{\,\big|\,}   \def\rset{\big\}}
\def\Lset{\Big\{\,}    \def\Mset{\,\Big|\,}   \def\Rset{\Big\}}
\outer\def\proclaim #1. #2\par{\medbreak \noindent{\bf#1.\enspace}{\sl#2}\par
  \ifdim\lastskip<\medskipamount
  \removelastskip\penalty55\medskip\fi}
\def\qed{\hfill{$\vcenter{\hrule height1pt \hbox{\vrule width1pt height5pt
   \kern5pt \vrule width1pt} \hrule height1pt}$} \medskip}
\def\low#1{{\lower1pt \hbox{$\scriptstyle #1$}}}
\def\downto{{\raise 1pt \hbox{$\scriptstyle \,\searrow\,$}}}
\def\tto{\;{\lower 1pt \hbox{$\rightarrow$}}\kern -12pt
           \hbox{\raise 2.8pt \hbox{$\rightarrow$}}\;}
\def\for{\hskip0.9pt|\hskip0.9pt}
 
\def\ball{{I\kern -.35em B}}

\def\for{\hskip0.9pt|\hskip0.9pt}
\def\bx{\bar x}  

\def\bu{\bar u} 
\def\by{\bar y}  \def\bp{\bar p} 
  \def\bz{\bar z} \def\bp{\bar p}
\def\dom{\mathop{\rm dom}\nolimits} \def\rge{\mathop{\rm rge}\nolimits}
\def\gph{\mathop{\rm gph}\nolimits}
 \def\inT{\mathop{\rm int}\nolimits}

\def\epsilon{\varepsilon}  \def\eps{\varepsilon}     \def\phi{\varphi}
\def\R{\mathh{R}}   \def\reals{\mathbb{R}}

\def\clm{\mathop{\rm clm}\nolimits}   
\def\reg {\mathop{\rm reg}\nolimits}

\def\mdot{\kern -.1em \cdot \kern -.1em}

\def\lset{\big\{\,}    \def\mset{\,\big|\,}   \def\rset{\big\}}
\def\Lset{\Big\{\,}    \def\Mset{\,\Big|\,}   \def\Rset{\Big\}}

\def\R{\mathbb{R}}
\newtheorem{definition}{Definition}[section]
\newtheorem{theorem}[definition]{Theorem}

\newtheorem{proposition}[definition]{Proposition}
\newtheorem{corollary}[definition]{Corollary}
\newtheorem{remark}[definition]{Remark}
\newtheorem{example}[definition]{Example}

\newcommand{\bea}{\begin{eqnarray}}
\newcommand{\eea}{\end{eqnarray}}
\newcommand{\bda}{\begin{eqnarray*}}
\newcommand{\eda}{\end{eqnarray*}}

\newcommand{\dd}{\mbox{\rm\,d}}

\begin{document}

\centerline{\bf \Large  Strong Metric  Subregularity of Mappings}

\smallskip

\centerline{\bf \Large in Variational Analysis and Optimization}
\bs
\bs

\centerline{ \bf  R. Cibulka$^1$, $\quad $
 A.  L. Dontchev$^{2}$ \ \  \text{ and } \ \  A. Y. Kruger$^3$ }

\bs \bs
\centerline{\sc Dedicated to the memory of Jonathan M.   Borwein}
\bs \bs

\begin{abstract}  Although the property of strong metric subregularity of set-valued mappings has been present in the literature under various names and with various (equivalent) definitions
for more than two decades, it has attracted much less
attention than its older ``siblings'', the metric regularity and the strong (metric) regularity.
The purpose of this paper is to show that the strong metric
subregularity shares the main features of these  two most popular regularity properties and  is not less instrumental in applications. We show that  the strong metric subregularity of a mapping $F$ acting between metric spaces is stable under perturbations of the form $f+F$, where $f$ is a function with
a small calmness constant. This result is parallel to  the Lyusternik-Graves theorem for metric regularity and to the  Robinson theorem for strong regularity, where the perturbations are represented   by a function $f$ with a small Lipschitz constant. Then we  study   perturbation stability
of the same kind  for mappings  acting between Banach spaces, where $f$ is not necessarily differentiable but admits a set-valued derivative-like approximation. Strong metric $q$-subregularity is also considered, where $q$ is a positive real constant appearing as exponent in the definition. Rockafellar's  criterion  for strong metric subregularity involving injectivity of the graphical derivative   is
 extended to   mappings acting in infinite-dimensional spaces.
A sufficient condition
for strong metric subregularity  is established in terms of surjectivity of  the Fr\'echet coderivative, and it is shown by a counterexample that surjectivity of  the limiting coderivative is not a sufficient condition  for this property, in general. Then various versions of Newton's method for solving generalized equations are considered including
inexact and semismooth methods, for which superlinear convergence is shown under strong metric subregularity. As applications to optimization, a characterization of the strong metric subregularity of the KKT mapping is obtained, as well as a radius theorem for the optimality mapping of a nonlinear programming problem. Finally,   an error estimate is derived for a discrete approximation in optimal control under strong metric subregularity of the mapping involved in the Pontryagin principle.
\end{abstract}
\bs


\bs


\vspace{2 cm}

\small{$^1$Department of Mathematics, Faculty of Applied Sciences, University of West Bohemia, Univerzitn\'\i\ 22, 306 14  Pilsen, Czech Republic, cibi@kma.zcu.cz. Supported by the project GA15-00735S.

$^2$Mathematical Reviews, 416 Fourth Street,  Ann Arbor, MI 48107-8604, USA, ald@ams.org;   Institute of Statistics and Mathematical Methods in Economics,
Vienna University of Technology, Wiedner Hauptstrasse 8, A-1040, Austria. Supported by
NSF, grant 1562209,  the Austrian Science Foundation
(FWF),  grant
P26640-N25,
and the Australian Research
Council,
project DP160100854.

$^3$ Centre for Informatics and Applied Optimization, Federation University Australia, POB 663, Ballarat, VIC 3350, Australia, a.kruger@federation.edu.au. Supported by the Australian Research Council, project DP160100854.}

\newpage

\section{Introduction}

There are three basic properties of linear mappings in  analysis  and topology: surjectivity, injectivity and invertibility. Specifically, a linear and bounded mapping
$A$ acting from a Banach space $X$ to a Banach space $Y$ is  said to be surjective when for every $y \in Y$ there exists $x \in X$ such that $Ax=y$;
it is said to be injective when $Ax  = 0$ implies $x=0$; it is   said to be invertible when  for every $y \in Y$ there exists a unique  $x \in X$ such that $Ax=y$.
The combination of surjectivity and injectivity implies invertibility and  in this case the inverse mapping $A^{-1}$ is  linear and bounded. When $X=Y=\reals^n$ all three properties are equivalent. An extension of surjectivity  to nonlinear/set-valued
mappings which goes back to the Banach open mapping principle is the well-known property of {\em metric regularity}, a name
coined
by Borwein in \cite{Bor}. An extension of  invertibility, which is particularly useful in optimization,
is known as {\em strong metric regularity}, a property introduced  by Robinson in
 \cite{Rob}. In this paper we
 focus  on an extension of  injectivity to nonlinear/set-valued mappings  called {\em strong metric subregularity}, for which in this paper  we also  use the name   ``strong subregularity" for short.
 Although this property  has been present in the literature under various names
and with various (mostly equivalent) definitions for more than two decades, it has attracted much less
attention than its older ``siblings'', the metric regularity and the strong (metric) regularity.
The purpose of this paper is to demonstrate that the strong subregularity shares the main features of the other two regularity properties and  is not less instrumental in applications.

To put the stage, let us first fix the  notations and terminology. Throughout, $X$ and $Y$  are metric spaces in general  and any metric is denoted by $\rho(\cdot, \cdot)$.  The space $Y$ also appears as  a linear metric space with shift invariant metric, that is, a
metric with the property that $\rho(y_1+y,y_2+y)=\rho(y_1,y_2)$ for all $y_1,y_2,y\in Y$.  Both $X$ and $Y$
could also be  Banach spaces and this is always explicitly stated or
clear from the context.
A norm is generally denoted by  $\|\cdot\|$, sometimes with a subscript indicating a specific space.  The $n$-dimensional Euclidean space is denoted by $\reals^n$ and the set of nonnegative integers is denoted by $\N$.
 The distance from a point $x$ to a set $A$ in  a metric space is
${d(x, A)} = \inf_{y \in A}\rho(x,y)$; the distance to the empty set is always $+\infty$.
 The closed ball  centered at $x$ with radius $r$ is denoted by $\ball_r(x)$ and the closed  unit ball is $\ball$. A set $U$ is said to be a neighborhood of a point $x$ when there exists a real $r>0$ such that $\ball_r(x) \subset U$.

A set-valued mapping $F$ acting from $X$ to the subsets of $Y$, denoted $F:X\tto Y$, is associated with  its graph $\,\gph F = \lset
(x,y)\in X\times Y\mset y\in F(x)\rset$, its  domain  $\,\dom F = \lset
x \in X \mset F(x) \neq \emptyset \rset$ and its range $\,\rge F = \lset
y \in Y \mset \exists \, x \in X \text{with} y \in F(x)\rset$. The  inverse
of $F$ is defined as $y \mapsto F^{-1}(y) = \lset {x \in X} \mset
y\in F(x)\rset$.  The space of all linear bounded (single-valued) mappings acting between Banach spaces  $X$ and $Y$ and   equipped with the standard operator norm is  denoted by ${\mathcal L}(X,Y)$.  A mapping $H$ acting between Banach spaces $X$ and $Y$ is said to be positively homogeneous when its graph is a cone. For a  positively homogeneous mapping $H:X\tto Y$ the expression
$\sup_{\|x\|\leq 1} d(0,H({x}))$ is said to be the { inner norm} of $H$   and denoted by
$\|H\|^-$, while  the expression
$\sup_{\|x\|\leq 1}\sup_{y\in H(x)}\|y\|$ is  the { outer norm} of $H$ and denoted by
$\|H\|^+$. Also, recall that the {\em measure of non-compactness} \cite{AKPRS}  of a set $\mathcal{A}$ is  defined as
$$
 \chi (\mathcal{A}) = \inf\bigg\{  r > 0 {\mset} \ \mathcal{A} \subset \bigcup\Big\{ \ball_r(A) {\mset} \
A\in{\mathcal B}\Big\}, \ {\mathcal B}\subset  \mathcal{A} \ {\rm finite}\bigg\}.
$$

Given a (set-valued) mapping $F$ acting from a metric space $X$ to (the subsets of) a metric space $Y$, a point $(\bx,\by) \in \gph F$ and neighborhoods $U$ of $\bx$ and $V$ of $\by$, the submapping $U \ni x \mapsto F(x)\cap V$   is said to be  a {\em graphical  localization at $\bx$ for $\by$}.  Local invertibility of $F$ at $(\bx, \by)$   is identified with $F^{-1}$  having  a localization at $\by$ for $\bx$ which is  single-valued (a function).
The most known manifestation of invertibility of a (nonlinear) function is the classical inverse function theorem: the inverse $f^{-1}$  of
a strictly differentiable at $\bx$  function
$f:X\to Y$ between Banach spaces has a strictly differentiable at $f(\bx)$   single-valued localization at $f(\bx)$ for $\bx$ if and only if
the strict derivative $Df(\bx)$   is invertible. For a general mapping $F:X\tto Y$,  the property  that $F^{-1}$ has a  Lipschitz continuous  single-valued localization at $\by$ for $\bx$ is known as  {\em strong metric regularity} of $F$ at $\bx$ for $\by$. In this paper we also use the shorter name {\em strong regularity}  as in Robinson's original definition in \cite{Rob} which, strictly speaking, is somewhat different but is based on the same idea.

A mapping $F:X \tto Y$
is said to be  {\em   metrically regular} at $\bx$ for $\by$ when $\by \in F(\bx)$,  $\gph F$ is locally closed at $(\bx,\by)$,
meaning that there exists a  neighborhood $W$ of $(\bx, \by)$ such that the set $\gph F\cap W$ is closed in $X \times Y$, and   there is a constant $\kappa \geq 0$ along  with
neighborhoods $U$ of $\bx$ and $V$ of $\by$ such that
\be \label{mr}
      d \big(x, F^{-1}(y) \big) \leq \kappa d \big(y, F(x)\big)
      \quad \text{for every }  (x,y) \in U \times V.
\ee
The infimum of $\kappa \geq 0$ for which there exist neighborhoods $U$ and $V$ such that \eqref{mr} holds is called the regularity modulus of $F$
 and denoted $\reg(F;\bx\for \by)$.  We use the convention that $\reg(F;\bx\for\by) < +\infty$ if and only if  $F$ is metrically regular at $\bx$ for $\by$. A  mapping   $A  \in {\mathcal L}(X,Y)$ is metrically regular at any point if and only if it is surjective in which case $\reg A =  \|A^{-1}\|^-$; this comes from the Banach open mapping principle.  A mapping $F$ is strongly regular at $\bx$ for $\by$
if and only if $F$ is metrically regular at $\bx$ for $\by$ and the inverse $F^{-1}$ has a graphical localization at  $\by$ for $\bx$ which is nowhere multivalued; in this case for every  $\ell > \reg (F; \bx\for \by)$ there exists a neighborhood of $\by$ where  the localization is  Lipschitz continuous  with a Lipschitz
constant $\ell$.

A  generally set-valued  mapping $F$ acting from  a metric space $X$ to  the subsets
of a metric space $Y$
is said to be  {\em strongly metrically subregular} at $\bx$ for $\by$ when
$\by \in F(\bx)$ and there is a constant $\kappa \geq 0$ along  with
neighborhoods $U$ of $\bx$ and $V$ of $\by$ such that
\be\label{sr}
 \rho(x, \bx)   \leq \kappa d(\by, F(x)\cap V)
      \; \text{for all} x \in U.
\ee
This property can be equivalently defined, see \cite[Section 3I, p. 194]{book} with just one neighborhood $U$ by adjusting its size, as follows:
there is a constant $\kappa \geq 0$ along  with  a
neighborhood $U$ of $\bx$ such that
\be \label{sr1}
 \rho(x, \bx)  \leq \kappa d(\by, F(x))
      \; \text{for all} x \in U.
\ee
Either definition  yields that $\bx$ is the only point in $U$ such that $\by \in F(\bx)$; that is, $\bx$ is an {\em isolated point} of $F^{-1}(\by)$.
The infimum of $\kappa \geq 0$ over neighborhoods $U$ and $V$ such that \eqref{sr} holds
(or over $U$  such that \eqref{sr1} holds) is called the subregularity modulus of $F$
 and denoted by $\subreg(F;\bx\for \by)$. We adopt the convention that $\subreg(F;\bx\for\by) =+\infty$ whenever  $F$ is not strongly subregular at $\bx$ for $\by$.
Note that we do not assume that the graph of $F$ is locally closed at the reference point in the definition of strong subregularity.  A mapping   $A\in {\cal L}(X,Y)$ whose range is closed  is strongly subregular everywhere if and only if it is injective; in this case
$\subreg A =  \|A^{-1}\|^+$; note that in finite dimensions the range of a linear bounded mapping is always closed.

There is a close connection between strong metric subregularity and the properties of the distance function $(x,y) \mapsto d(y, F(x))$, see \cite[Theorem 3I.5]{book}. Directly from the definition it follows that  a set-valued mapping $F: X \tto Y$ is strongly subregular at $\bx$ for $\by$ if and only if $\bx$ is a local sharp minimizer of the function $x \mapsto d(\by, F(x))$. Recall that  a point $\bx \in \dom \varphi$ is called a {\it local sharp minimizer} of  a function $\varphi: X \to \reals\cup \{+ \infty\}$
whenever  there is a neighborhood $U$ of $\bx$ and a constant $\beta > 0$ such that
$
 \varphi (x)  \geq \varphi (\bx) + \beta \rho(x, \bx)  \;\text{for all}  x \in U.
$

A mapping $F:X\tto Y$ is strongly subregular at
$\bx$ for $\by$  if and only if its inverse $F^{-1}$
has the so-called {\em isolated calmness property} at $\by$ for $\bx$.
Specifically, whenever $F$ is strongly subregular at $\bx$ for $\by$   there exist a constant $\mu \geq 0$ and neighborhoods $U$ of $\bx$ and $V$
of $\by$ such that
\be\label{isolc}
  F^{-1}(y)\cap U \,\subset\,
  \ball_{\mu \rho(y,\by)}(\bx)  \;\text{for all} y\in V.
\ee
Moreover, the infimum of all $\mu$ such that this inclusion
holds for some neighborhoods $U$ and
$V$, which we denote as $\clm (F^{-1};\by\for\bx)$,  equals $\,\subreg (F;\bx\for\by)$.  The proof of this statement is straightforward, see e.g. \cite[Theorem 3I.3]{book} where it is stated in finite dimensions but can be easily translated into the language of metric spaces.

Strong subregularity and isolated calmness
have been considered in various contexts and under various names in
the literature.
Isolated calmness was formally introduced by the second author in \cite{don95} under the name ``local upper Lipschitz continuity at a point"; in the same paper the perturbation  stability of this property
 was first proved. The equivalent property of strong subregularity
was considered earlier, without giving it a name, by Rockafellar~\cite{R89}. The name ``strong metric subregularity" was first used in \cite{regcon} where its equivalence with the isolated calmness was proved.

In finite dimensions there is a class of strongly subregular mappings with a particularly simple description.   The following   theorem is based on   an important result   by Robinson
 \cite{Rob1}:

\begin{theorem}\label{poly} Consider a mapping $F:\reals^n\tto\reals^m$ whose graph is the union of finitely many  polyhedral convex sets. Then $F$ is strongly  subregular at $\bx$
for $\by$ if and only if $\bx$ is an isolated point of $F^{-1}(\by)$.
\end{theorem}

The strong subregularity obeys  the paradigm of the inverse function theorem, by which we mean that the property   is stable (persistent)  under addition of  a function whose calmness  constant is smaller than  the reciprocal of the subregularity modulus. The metric regularity and the strong regularity also obey this paradigm but when the function added to the mapping has  a Lipschitz constant smaller than the reciprocal of the regularity modulus. In the case when the mapping is represented by a strictly differentiable function this yields that all three  properties are preserved  under linearization.

If we fix $y=\by$ in the definition of metric regularity \eqref{mr} we obtain the property of {\em metric subregularity:}
\be \label{msr}
      d \big(x, F^{-1}(\by) \big) \leq \kappa d \big(\by, F(x)\big)
      \quad \text{for every }  x \in U.
\ee
In contrast to metric regularity, the property \eqref{msr} does not obey the paradigm of the inverse function theorem, as explained in  
\cite[Section 3.8]{book}.
Indeed, from Theorem~\ref{poly}  every linear  mapping between $\reals^n$ and
$\reals^m$ is metrically subregular, but not every smooth function has this property.  Nevertheless,
  for some special
kinds of mappings one may expect stability criteria in terms of infinitesimal approximations, see \cite{gfr}.

 The following proposition puts together the strong regularity, the metric regularity, and the strong subregularity of a function  $f$ at $\bx$ against the invertibility, surjectivity and injectivity  of its strict derivative
$Df(\bx)$.  With some abuse of notation, for a function $f$  we say that $f$ is (strongly) metrically (sub)regular at $\bx$ and  write
(sub)$\reg(f;\bx)$ instead of  (sub)$\reg(f;\bx\for f(\bx))$.

\begin{proposition}\label{propo}
Let $X$ and $Y$ be Banach spaces and let  $f:X\to Y$ be strictly differentiable at $\bx$. Then
\paritem{(i)} $f$ is strongly regular at $\bx$  if and only if  $Df(\bx)$ is invertible, in which case
$\reg(f;\bx) = \|Df(\bx)^{-1}\|;$
\paritem{(ii)} $f$ is  metrically regular at $\bx$  if and only if  $Df(\bx)$ is surjective, in which case $\reg(f;\bx) = \|Df(\bx)^{-1}\|^-;$
 \paritem{(iii)}  Suppose that $\rge Df(\bx)$ is closed. Then $f$ is strongly  {subregular} at $\bx$  if and only if  $Df(\bx)$ is injective, in which case $\subreg(f;\bx) = \|Df(\bx)^{-1}\|^+.$ Moreover, in this case it is sufficient to assume that $f$ is Fr\'echet differentiable at $\bx$.
\end{proposition}

The first statement
is a version of the classical inverse function theorem. The second statement
follows from the Lyusternik-Graves theorem. We will present a  general version of the third statement
in Section~\ref{pert} {where we also show that  in infinite dimensions  the assumption regarding the closedness of the range of the derivative mapping cannot be removed.}

From Proposition~\ref{propo} we obtain that if a smooth {function} is both strongly subregular and metrically regular at $\bx$, then
it is strongly   regular at $\bx$. This is not true however for set-valued mappings even if we require  strong subregularity around the reference point. As a counterexample, take $F(x) = \{-x, x\}, x \in \reals$, which is  both strongly subregular and  metrically regular at $0$ for $0$, strongly regular at every point
in its graph different from the origin, and not strongly  regular at $0$ for $0$.

In this paper we present a collection of new  results regarding strong metric subregularity; we  also
give  extended  versions of known results which  is
clearly indicated in the text. The paper has two main parts. The first part  presents theoretical results mostly related to stability of strong subregularity with respect to (derivative-type) approximations.   First we focus on showing perturbation stability  in general metric spaces and some consequences for differentiable functions and polyhedral mappings in finite dimensions. Then we deal with mappings of the form $f+F$ where $f$ is a not necessarily differentiable function
 and $F$ is a set-valued mapping. Section~\ref{Secq} shows extensions to the so-called strong $q$-subregularity. In Section~\ref{gd} a partial extension of Rockafellar's  criterion  for strong subregularity is obtained  for  mappings acting in infinite-dimensional spaces.
A sufficient condition
for strong subregularity  is established in terms of surjectivity of  the Fr\'echet coderivative, and it is shown by a counterexample that surjectivity of  the limiting coderivative cannot serve as a sufficient condition  for this property to hold.

The second  part of the paper
 is devoted to applications  that are the main motivation of
 this study. We consider first  various versions of Newton's method  including
inexact and semismooth methods, for which a specific mode of convergence is shown under strong subregularity. For a standard  nonlinear programming problem, a characterization of the strong subregularity of the  optimality mapping is obtained in terms of a strong form of the Mangasarian-Fromovitz constraint qualification and a quadratic growth condition for the objective function.
A related result  is obtained in  \cite{AAG2} for 
a proper lower semicontinuous convex function $g : X \to \reals \cup \{+\infty\}$ defined
on a Banach space $X$, whose dual is denoted by $X^*$. Namely, it is shown that the subdifferential mapping $\partial g: X \tto X^*$, understood in the sense of convex analysis, is strongly subregular at a point $(\bx, \bx^*) \in \gph \partial g$  if and
only if there exist positive constants $\beta$ and $\delta$ such that
$$
 g(x)  \geq g (\bx) + \langle \bx^*, x - \bx \rangle + \beta \|x-\bx\|^2 \quad \mbox{whenever} \quad x \in \ball_\delta(\bx),
$$
where $\langle \cdot, \cdot \rangle: X^* \times X \to \reals$ denotes the duality pairing.  Generalizations of the above results to a non-convex function $g$ by using limiting subdifferential and under appropriate additional assumptions can be found in \cite[Corollary 3.3 and 3.5]{DMN}, see also \cite{WW}. If $X=\reals^n$, a relation of strong subregularity of the limiting subdiferential and quadratic growth of a semi-algebraic function $g$ can be found in \cite[Theorem 3.1]{DI}.          

As another application,  a radius theorem for the optimality mapping for a nonlinear programming problem  is proven, giving an expression  for the minimal perturbation of the objective function by a quadratic form for which the second-order sufficient optimality condition is violated. Finally,  an error estimate is derived for a discrete approximation in optimal control under strong subregularity of the mapping involved in the Pontryagin principle.

\section{Perturbed strong subregularity}\label{pert}

Recall \cite[Section~1.3]{book}
that a function  $g$ acting between metric spaces $X$ and $Y$   is said to be calm at
$\bx$ when $\bx \in \dom g$ and
there exist
a neighborhood $U$ of $\bx$ and a constant $\mu \geq 0$
such that
\be\label{calm}
\rho(g(x), g(\bx))\leq \mu\rho(x,\bx) \quad \text{for every} \quad x \in U \cap\dom g.
\ee
The infimum of $\mu \geq 0$ such that \eqref{calm} holds for some neighborhood $U$ of $\bx$ is the {\it calmness modulus} of $g$ at $\bx$ and is denoted by $\clm(g;\bx)$.
Note that $\bx$ does not have to be an interior point of $\dom g$.

The following theorem shows that the strong subregularity obeys the paradigm of the inverse function theorem: the property is preserved under perturbations by a function with a small calmness modulus. A version of it appeared first in \cite[Theorem 3.2]{don95} and was echoed later in other publications. More recently,  \cite[Theorem 3I.7]{book} uses an equivalent definition of strong subregularity and is given in finite dimensions, while the proof in   \cite[Theorem 3.2]{u} uses  the notion of the steepest displacement rate of a set-valued mapping. The proof given here is just an application of the definitions; we present it for completeness.

\begin{theorem}\label{main}
Suppose {that}  $X$ is a metric space and $Y$ is a linear metric space with shift invariant metric.
Let  {$a$,  $\kappa$, and $\mu$} be positive constants such that $\kappa \mu < 1$. Consider  a mapping $G:X \tto Y$ which is strongly subregular at $\bx$ for $\by$ with a constant $\kappa$ and a neighborhood $\ball_a(\bx)$, and  a function $g:X\to Y$ which  is calm
at $\bx$ {with a constant $\mu$ and a neighborhood $\ball_a(\bx)$.}
Then $g+G$ is  strongly subregular at $\bx$ for $\by+g(\bx)$ with the constant $(\kappa^{-1}-\mu)^{-1}$  and the  neighborhood ${\ball_a(\bx)}$; in particular
$$\subreg(g+G;\bx\for \by+g(\bx)) \leq \frac{\kappa}{1-\kappa\mu}.$$
\end{theorem}

\begin{proof} By assumption, we have
\be\label{st}
      \rho(x,\bx) \leq \kappa d(\by, G(x)) \text{ and } \rho(g(x), g(\bx))\leq \mu\rho(x,\bx)
      \; \text{for all} x \in {\ball_a(\bx)} \cap\dom g.
\ee
Observe that $\dom(g+G)=\dom g\cap\dom G$.
Take any $x \in {\ball_a(\bx)} \cap\dom g$ and any $z \in g(x)+G(x)$ (if there is no such $z$ we have $d(\by + g(\bx), g(x)+G(x))= {+} \infty$ and there is nothing to prove). Then there exists $y\in G(x)$ such that $y=z-g(x)$ and from \eqref{st}
we get
\bas
\rho(x,\bx) & \leq & \kappa d(\by, G(x)) \leq \kappa \rho(\by,y) = \kappa \rho(\by, z-g(x))  \\
& \leq & \kappa \rho(\by, z -g(\bx))+ \kappa \rho(g
(x), g(\bx)) \leq \kappa\rho(\by, z -g(\bx)) + \kappa\mu\rho(x, \bx).
\eas
Taking into account that $\kappa\mu< 1$ and  $z$ is an arbitrary point in $g(x)+G(x)$, we obtain
$$ \rho(x,\bx) \leq \frac{\kappa}{1-\kappa\mu}d\big(\by+g(\bx), (g+G)(x)\big).$$
The proof is complete.
\end{proof}

{The above statement fails when the perturbation $g$ is represented by a (calm) set-valued mapping even for $X=Y=\reals$. Indeed, the mapping  $G(x) = \{1+x^2, 2x\}$ is strongly subregular   at $0$ for $0$. Let
$g(x) = \{-1, -x\}$; clearly $g$ has the isolated calmness property at $0$ for $0$. However, as easily seen,  the sum
$ g(x)+G(x)=\{x^2, 1-x+x^2, 2x-1, x\}$ is not strongly subregular   at $0$ for $0$.}

The following corollary specifies the result in Theorem~\ref{main} for the case when
  the (single-valued) function is approximated  by another  function.
\begin{corollary} \label{prop1}
Suppose {that} $X$ is a metric space and $Y$ is a linear metric space with shift invariant metric.
Consider $F : X \tto Y$, a point $(\bx,\by) \in \gph F$ and two functions $f : X \to Y$ and $h : X \to Y$ with
$\bx \in \dom f \cap \dom  h$.  Suppose that  $h + F$ is strongly subregular at $\bx$ for $h(\bx) + \by$,
{the difference} $f-h$ is calm at $\bx$,
and $$ \subreg( h +F; \bx \for h(\bx) + \by) \cdot \clm(f - h; \bx)  < 1.$$ Then
the mapping $f +F$ is strongly subregular at $\bx$ for $f (\bx)+ \by $ and
$$
\subreg( f +F; \bx {\for f(\bx) + \by}) \leq \frac{\subreg( h +F; \bx \for h(\bx) + \by)}{1 - \subreg( h +F; \bx \for h(\bx) + \by) \cdot \clm(f - h; \bx) }.
$$
In particular, if $\clm(f - h; \bx) = 0$, then the mapping $f +F$ is strongly subregular at $\bx$ for
$f(\bx)+ \by$ if and only if $h+F$ is strongly subregular at $\bx$ for $h(\bx)+ \by$, in which case
$$
 \subreg( f +F; \bx \for f (\bx)+ \by) = \subreg(h+F; \bx \for h(\bx)+ \by).
$$
\end{corollary}
\begin{proof} To show the first statement, fix any $\kappa > \subreg( h +F; \bx {\for h(\bx)+\by})$ and $\mu > \clm(f - h; \bx)$ such that $\kappa \mu < 1$. Clearly, there is $a> 0$ such that the assumptions of Theorem~\ref{main} hold for $G = h+F$ and {$g = f - h$}. Hence $f+F = g+G$ is strongly subregular at $\bx$ for $f (\bx)+ \by $ with modulus not greater than $\kappa/(1-\kappa\mu)$.
The second statement follows from the first one and the fact that  $f$ and $h$ can be interchanged.
\end{proof}

\begin{remark} \rm
When $X$ and $Y$ are Banach spaces and  $f : X \to Y$ is Fr\'echet differentiable at $\bx \in X$ then the function $x\mapsto h(x):= f(\bx) + Df(\bx)(x - \bx)$
satisfies the conditions in the second part of Corollary~\ref{prop1}. Taking $F \equiv 0$ we arrive at Proposition~\ref{propo} (iii). But we can consider the  much larger class of semidifferentiable functions. Recall that a function $f : X \to Y$ is called semidifferentiable at $\bx$, if there is a (unique) continuous and positively homogeneous function $\varphi: X \to Y$ such that the function $h:=f(\bx) + \varphi(\cdot-\bx)$ is the first-order approximation to $f$ at $\bx$, that is,  $\clm(f - h; \bx) = 0$. Every piecewise smooth function $f: \reals^n \to \reals^m$ is semidifferentiable at any interior point of its domain \cite[Proposition 2D.8]{book}. Also if  $f: \reals^n \to \reals^m$ is locally Lipschitz at $\bx$, then $f$ is semidifferentiable at $\bx$ if and only if $f$ is directionally differentiable at $\bx$ \cite[Proposition~2D.1]{book}.
\end{remark}

 \begin{remark}\rm
 Let $f: X \to Y$, with $X$ and $Y$ being normed spaces, and  $\bx \in X$ be such that there is a positively homogeneous function $\varphi: X \to Y$ which is continuous at $0$ and $\clm(f - \varphi(\cdot - \bx); \bx) < \varepsilon$ for some positive $\varepsilon$ (such a {function} $\varphi$ is called the first-order $\varepsilon$-approximation of $f$ at $\bx$ in \cite{u}). Taking $F\equiv0$ and observing that $h:=f(\bx) + \varphi(\cdot - \bx)$ is strongly subregular at $\bx$ if and only if so is $\varphi$ at $0$,  we get \cite[{Theorem 4.1}]{u}: If $\varphi$ is strongly subregular at $0$ and $ \varepsilon \subreg(\varphi;0) < 1$, then $f$ is strongly subregular at $\bx$ with modulus not greater than $\subreg(\varphi;0)/(1 - \varepsilon\subreg(\varphi;0))$.
\end{remark}

We present next a theorem regarding  perturbation stability of strong subregularity
 in an implicit function form. It is an  infinite-dimensional version of  \cite[Theorem 3I.14]{book} whose proof also works in this case with  a few minor adjustments and therefore will not be reproduced here.

\begin{theorem} Let $X$, $P$ and $Y$ be Banach spaces and let $f:P\times X\to Y$ and $F:X\tto Y$.
Consider the generalized equation $f(p,x)+F(x)\ni 0$, its solution mapping
$
   S: p\mapsto \lset x \mset f(p,x)+F(x)\ni 0 \rset,
$
 and a pair $(\bp,\bx) \in \gph S$, and suppose that $f$ is continuously Fr\'echet differentiable  on a neighborhood of
$(\bp, \bx)\in \inT\dom f$. If the mapping
$$
    h+F \quad \text{for} \quad  h= f(\bp,\bx) +D_x f(\bp,\bx)(\cdot-\bx)
$$
is strongly subregular at $\bx$ for $0$, then $S$ has the
isolated calmness property at $\bp$ for $\bx$ with
\be\label{ess}
   \clm(S;\bp\for\bx) \, \leq \,
                \subreg(h+F;\bx\for 0)\cdot \|D_p f(\bp,\bx)\|.\ee
Furthermore, when  $P$ and $Y$ are Hilbert spaces and
$D_p f(\bp, \bx)$ is surjective,
then the converse implication holds as well: the mapping $h+F$ is
strongly subregular at $\bx$ for $0$ provided that $S$
has  the isolated calmness property at $\bp$ for $\bx$.
\end{theorem}

\begin{proof} The proof of the first part   of the theorem which gives the estimate
\eqref{ess} is identical with the proof of \cite[Theorem 3I.13]{book}
with general Banach space norms replacing the Euclidean ones.
Consider the mapping
$$\Psi:(x,y) \mapsto \{p  \mid f(p,x) - h(x) + y=0\} \text{for} (x, y) \in X\times Y.$$
Let $A= D_p f(\bp, \bx):P \to Y$. Since $P$ and $Y$ are Hilbert spaces,
the mapping $AA^*: Y \to Y$, where $A^*$ is the adjoint to $A$,  has a linear bounded inverse. Let $c= \|A^*(AA^*)^{-1}\|$. The further proof is identical to the proof of  \cite[Lemma 2C.1]{book}. To finish, use the argument in the proof of \cite[Proposition 3I.15]{book}  replacing the Euclidean norms  by the norms of $X$, $Y$ and $P$ spaces, respectively.
\end{proof}

Generalizations of the first part of the above statement for parametric generalized equations with a nonsmooth single-valued part can be found in \cite[Section 5]{u} (cf. Theorem~\ref{thmSSSR2} in the next section). Combining Corollary~\ref{prop1} and Theorem~\ref{poly} we obtain the following result:

\begin{theorem}\label{2} Let $X$ and $Y$ be Banach spaces.
Consider a function $f:X\to Y$ which is {Fr\'echet} differentiable at a point $\bx \in X$ and a set-valued mapping $F:X\tto Y$ with
$(\bx, \by) \in \gph F$. Then the mapping $f+F$ is strongly {subregular} at $\bx$ for $f(\bx)+\by$ if and only if
the mapping {$H:= f(\bx) + Df(\bx)(\cdot-\bx) +F$} has the same property. In the case when $X=\reals^n$, $Y=\reals^m$ and the graph of $F$  is the union of finitely many  polyhedral convex sets, the mapping $H$, and hence
$f+F$, is strongly {subregular} at $\bx$ for $f(\bx)+\by$ if and only if $\bx$ is an isolated point of
 $H^{-1}(f(\bx) + \by)$.
\end{theorem}

{Theorem~\ref{2} yields the statement (iii) in Proposition~\ref{propo} but note that the latter 
imposes the additional condition that the range of $Df(\bx)$ is closed. Indeed, $f$ is strongly subregular
at $\bx$ if and only if the linearization $f(\bx)+Df(\bx)(\cdot - \bx)$ has the same property.
The problem is that an injective  linear and bounded mapping is not necessarily   strongly subregular.
Let's have a closer look   at that.}

{ By linearity, $A\in {\cal L}(X,Y)$ is strongly subregular everywhere if and only if $A$ is strongly subregular at $0$ for $0$.  From \eqref{sr1} we
obtain  that $A$ is strongly subregular at $0$ for $0$ if and only if  
 \begin{equation} \label{srlinear}
    \liminf_{x \to 0, x \neq 0} \frac{\|Ax\|}{\|x\|} = \inf_{\|h\| = 1} \|Ah\| > 0.
 \end{equation}
If the dimension of $X$ is finite, then \eqref{srlinear} holds if and only if $A^{-1}(0) = \{0\}$, that is, $A$ is injective. This is not true in general as} {Example~\ref{exLinear} shows. However, if an operator $A\in {\cal L}(X,Y)$ has a closed range  then the Banach open mapping theorem yields that there is  a constant $\kappa > 0$ such that for any $y \in \rge A $ there is $x \in X$ such that $y=Ax$ and $\|x\| \leq \kappa \|y\|$. Then the injectivity of $A$ implies that such a point $x$ is unique and therefore   
$$
 \|x\|  \leq \kappa \|Ax\| \quad \mbox{for any} \quad x \in X. 
$$
Consequently, any bounded linear operator which is injective and has a closed range is strongly subregular at $0$ for $0$, and hence strongly subregular everywhere.}

\ms

\begin{example} \label{exLinear} \rm { {Let  $X=\ell_\infty$, the space of (infinite) sequences  $\{x_k\}$  in $\reals$ equipped with the norm $\|\{x_k\}\|_\infty = \sup_{k \in \N}|x_k|$, and $Y=\ell_2$, the space of (infinite) sequences  $\{x_k\}$  in $\reals$ equipped with the norm $\|\{x_k\}\|_2 = \sqrt{\sum\limits_{k =1}^{+\infty} (x_k)^2}$.} Define the operator $A$  by 
$$
 A(\{x_k\}) = \{k^{-1}x_k\}_{k=1}^{+\infty}, \quad \{x_k\} \in \ell_\infty.
$$
Then $A \in \mathcal{L}(\ell_\infty, \ell_2)$ with $\|A\| = \frac{\pi}{\sqrt{6}}$. {Indeed, letting} $x_k:=1$, $k \in \N$, we get $\|\{x_k\}\|_{\infty}=1$ and $\|A(\{x_k\})\|_{2}^2 = \|\{1/k\}\|_2^2 = \sum\limits_{k =1}^{+\infty} (1/k)^2=\pi^2/6$. On the other hand, for any $\|\{x_k\}\|_{\infty}\leq 1$ and $\{y_k\}:=A(\{x_k\})$  we  have $ 0 \leq (y_k)^2 =k^{-2} (x_k)^2 \leq k^{-2}$ for any $k \in \N$, which means that $\|\{y_k\}\|_{2}^2 \leq \pi^2/6$. The mapping $A$ is injective, but not strongly subregular at $0$ for $0$. Indeed, suppose on the contrary  that there are
$\kappa > 0$ and $a > 0$ such that 
$$
 \|\{x_k\}\|_\infty  \leq \kappa \|A(\{x_k\})\|_2       \; \text{for all} \{x_k\} \in a \ball_{\ell_\infty}.
$$
Pick any $n \in \N$ such that $n > \kappa$ and then set $x_k = a$ if $k=n$ and $x_k = 0$ otherwise. Then $\|\{x_k\}\|_\infty =a$ and $\|A(\{x_k\})\|_2 = a/n$. Thus 
$$
 a \leq \kappa \frac{a}{n} < a,
$$
a contradiction.} {Given $n \in \N$,  let $x_{k,n} = 1$ if $k=n$ and $x_{k,n} = 0$ otherwise. Then $x_n:=\{x_{k,n}\} \in \ell_\infty$ is such that $\|x_n\|_\infty =1$ and $\|Ax_n\|_2 = 1/n$. Hence $\inf_{\|\{x_k\}\|_\infty = 1} \|A(\{x_k\})\|_2 = 0$, that is, \eqref{srlinear} fails.}
{{The range of $A$ is not closed.} Indeed, given $n \in \N$,  let $y_{k,n} = k^{-2/3}$ if $k \leq n$ and $y_{k,n} = 0$ otherwise; then $y_n:=\{y_{k,n}\} \in \ell_2$. For each $n\in \N$,  if we set $x_{k,n} = k^{1/3}$ if $k\leq n$ and $x_{k,n} = 0$ otherwise, then $x_n:=\{x_{k,n}\} \in \ell_\infty$ and $Ax_n = y_n$. Then $y:=\lim_{n \to +\infty}y_n = \{k^{-2/3} \} \in \ell_2$ but $x = A^{-1}y=\{k^{1/3}\} \notin \ell_\infty$.
}
\end{example}

\section{Set-valued derivative-type approximations}\label{svap}

In this section we continue the analysis started in the preceding section of  mappings of the form $f+F$, where now $f$ is a function which is  calm at the reference point but not {necessarily} differentiable there, and $F$ is a  set-valued mapping.
We will now approximate the {possibly nonsmooth} function $f$ around the reference point  by a set  $\mathcal{A}$  in  ${\mathcal L}(X,Y)$. {This approach goes back to \cite{io} and the concept of a prederivative which is generated by a set of linear operators.}

\begin{theorem} \label{thmSSSR}
Let $X$ and $Y$ be Banach spaces and
consider a function  $f:X\to Y$, a  mapping  $F:X \tto Y$ and a point
 $(\bx,\by) \in X \times Y$ such that $\by \in f(\bx) + F(\bx)$. Suppose that  there exist a subset $\mathcal{A}$ of ${\mathcal L}(X,Y)$ and a constant  $c> 0$ such that:
 \paritem{(i)} there is
a constant {$r>0$} such that for every $u\in \ball_r(\bx)$ one can find $A \in \mathcal{A}$ satisfying
\be\label{c}
  \|f(u) - f(\bx) - A(u - \bx)\| \leq c \|u - \bx\|;
\ee
 \paritem{(ii)} for every  $A\in \mathcal{A}$  the mapping
  \be\label{HA}  X \ni x \mapsto H_A(x):=f(\bx) + A(x-\bx) + F(x) \ee
is strongly subregular at $\bx$ for $\by$  and
\be\label{m}
(c + \chi (\mathcal{A})) \cdot  m  < 1,
\ee
where $$m: =\sup\limits_{A \in \mathcal{A}} \subreg(H_A; \bx \for {\by}).$$
 Then  $f+F$ is strongly subregular at $\bx$ for $\by$; moreover
 \be\label{smse} \subreg(f+F; \bx \for {\by}) \leq \frac{m}{1- (c + \chi (\mathcal{A})) \cdot m}.\ee
\end{theorem}

\begin{proof} Note that from \eqref{c} we have $\bx \in \inT\dom f$ and also \eqref{m} yields  that $m < {+} \infty$.
Choose  $\kappa > m$ and $\gamma > 0$ such that
$$
 (c + \chi (\mathcal{A}) + \gamma) \kappa < 1.
$$
Let $r$ be as in condition (i).  We will  show first   that there exists  $a \in (0,r]$  such that
 \begin{equation}   \| x - \bx\| \leq \frac{\kappa}{1 - \kappa (\chi (\mathcal{A}) + \gamma)} d\big({\by}, H_A(x)\big) \quad \mbox{whenever} \quad x \in \ball_a(\bx)
   \quad \mbox{and} \quad A\in \mathcal{A}.
 \label{bbbs}
 \end{equation}
By the definition of $\chi (\mathcal{A})$, there is a finite set ${\mathcal{B}} \subset  \mathcal{A}$ such that
\begin{equation} \label{AsAf}
 \mathcal{A}  \subset {\mathcal{B}} + (\chi (\mathcal{A}) + \gamma)  \ball.
\end{equation}
Pick any $\tilde{A} \in {\mathcal{B}}$.  Then there exists $\alpha_{\tilde{A}} > 0$ such that
$$
 \| x - \bx\| \leq \kappa d\big({\by}, {H}_{\tilde{A}}(x)\big) \quad \mbox{whenever} \quad x \in \ball_{\alpha_{\tilde{A}}}(\bx ).
$$
Let $A' \in (\chi (\mathcal{A}) + \gamma) \ball$.   Since
$
 {H}_{\tilde{A}+A'} = {H}_{\tilde{A}} + A'(\cdot - \bar{x})
$, Theorem~\ref{main} implies that
$$
  \| x - \bx\| \leq  \frac{\kappa}{1 - \kappa (\chi (\mathcal{A}) + \gamma)} d\big({\by}, {H}_{\tilde{A} + A'}(x)\big) \text{ for every}   x \in  \ball_{\alpha_{\tilde{A}}}(\bx ).
$$
Thus, for any $\tilde{A} \in {\mathcal{B}}$ there is $\alpha_{\tilde{A}} > 0$ such that for each $A' \in (\chi (\mathcal{A}) + \gamma)  \ball$ the above inequality holds.  Let $a = \min \left\{ r, \min_{\tilde{A} \in {\mathcal{B}}} \alpha_{\tilde{A}}\right\}$.
 Taking into account \eqref{AsAf}, we obtain \eqref{bbbs}.

Choose any $x \in  \ball_a(\bx)$, then use (i) to find $A \in \mathcal{A}$ such that  \eqref{c} is satisfied.
Then  \eqref{c} along with
 \eqref{bbbs} gives us
\begin{eqnarray*}
  \| x - \bx\| &\leq & \frac{\kappa}{1 -  \kappa(\chi (\mathcal{A}) + \gamma)} d\big({\by}, {H}_{{A}}(x)\big) = \frac{\kappa}{1 -  \kappa(\chi (\mathcal{A}) + \gamma)} d\big({\by}-f(\bx) - A(x-\bx), F(x)\big)  \\
    &\leq  & \frac{\kappa}{1 -  \kappa(\chi (\mathcal{A}) + \gamma)}  \left(d\big({\by}-f(x), F(x)\big) + \| f(x) - f(\bx) - A(x - \bx)\| \right) \\
     &\leq  &  \frac{\kappa}{1 -  \kappa(\chi (\mathcal{A}) + \gamma)}  d\big({\by}, f(x)+ F(x)\big) + \frac{\kappa c}{1 -  \kappa(\chi (\mathcal{A}) + \gamma)}   \|x - \bx\|.
\end{eqnarray*}
Since $(c + \chi (\mathcal{A}) + \gamma) \kappa < 1$, we obtain
$$
  \| x - \bx\| \leq  \frac{\kappa}{1 -  \kappa(c + \chi (\mathcal{A}) + \gamma)} d\big({\by}, (f+ F)(x)\big).
$$
Thus, $f+F$ is strongly subregular at $\bx$ for ${\by}$. Since $\kappa > m$ and $\gamma > 0$ can be arbitrarily close  to $m$ and $0$, respectively, this yields \eqref{smse}.
\end{proof}

Let  $f: \reals^n \to {\reals^m}$ be  Lipschitz continuous around $\bx$. Bouligand's limiting Jacobian, denoted by  $\partial_B f(\bx)$, is defined as  the set of all matrices obtained as limits of the usual Jacobians $ \nabla f(x_k)$ for sequences $x_k \to \bx$ such that $f$ is differentiable at $x_k$. The convex hull of $\partial_B f(\bx)$ is Clarke's generalized Jacobian of $f$ at $\bx$  denoted by  $\partial_C f (\bx)$. If in Theorem~\ref{thmSSSR} we choose
$X=\reals^n$, $Y=\reals^m$,
and   $\mathcal{A}:=\partial_C f (\bx)$, then, as well {known}, see \cite[Proposition 6F.3]{book}, for every $c > 0$  there exists $r>0$ such that \eqref{c} is satisfied; that is, assumption (i) holds with an arbitrarily small $c>0$. In that case we also  have  $\chi(\partial_C f (\bx))=0$, and then Theorem~\ref{thmSSSR}
gives us   the following:
\begin{corollary} \label{c1}
Let $(\bx,\by) \in \reals^n \times {\mathbb{R}^m}$, $f:\mathbb{R}^n \to {\mathbb{R}^m}$ and  $F:\mathbb{R}^n {\tto} {\mathbb{R}^m}$ be such that $\by \in f(\bx) + F(\bx)$.
Suppose that $f$ is  Lipschitz continuous around $\bx$ and
for every  $A\in \partial_C f (\bx)$  the mapping  $H_A$ defined in \eqref{HA}
is strongly subregular at $\bx$ for $\by$.
 Then  $f+F$ is strongly subregular at $\bx$ for $\by$; moreover,
 $$\subreg (f+F; \bx \for \by) \leq \sup\limits_{A \in \partial_C f (\bx)} \subreg (H_A; \bx \for \by).$$
\end{corollary}

As an application of the above corollary, consider the inequality
\be\label{inequi}  f(x) \leq 0,\ee
where $f:\reals^n \to \reals^m$ is a Lipschitz continuous function around some $\bx \in \reals^n$.
Inequalities in $\R^m$ are understood componentwise.
Then, by combining Corollary~\ref{c1}  with Theorem~\ref{poly}, we obtain

\begin{corollary}\label{c11} In the context of the inequality system \eqref{inequi}, suppose that for every $A \in \partial_C f (\bx)$,
{the point}
$\bx$ is
the only solution of the inequality
  $$f(\bx) + A(x-\bx) \leq 0.$$
Then the mapping $f+\reals_+^m$ is strongly subregular at $\bx$ for $0$.
\end{corollary}

When $F$ is the zero mapping, from Corollary~\ref{c1}  we obtain an {analogue} of Clarke's inverse function theorem, which seems to be new:
\begin{theorem} Consider a function $f:\reals^n\to \reals^m$ which is Lipschitz continuous around
{$\bx \in \reals^n$}.
If all {matrices in}  the generalized Jacobian $\partial_C f (\bx)$ {have rank $n$} (which is only possible if {$n\le m$}),
then $f$ is strongly subregular at $\bx$.
\end{theorem}

In a different direction,  Theorem~\ref{thmSSSR}  may be extended in the following way:

\begin{theorem}\label{thmSSSR1}
Let $X$ and $Y$ be Banach spaces and
consider a function  $f:X\to Y$, a set-valued mapping  $F:X \tto Y$ and a point
 $(\bx,\by) \in X \times Y$ such that $\by \in f(\bx) + F(\bx)$. Suppose that there exist a mapping $\mathcal{H}: X \tto \mathcal{L}(X,Y)$ and a constant  $c> 0$ such that
 \paritem{(i)} there  is
a constant  $r > 0$ along with a selection $h$ for $\mathcal{H}$ such that
\begin{equation} \label{wsemi}
\|f(u)-f(\bx) - h(u)(u-\bx)\|\leq c \|u-\bx\|  \quad \mbox{whenever} \quad u \in \ball_r(\bx);
\end{equation}
 \paritem{(ii)} the assumption {\rm (ii)}
in Theorem~\ref{thmSSSR} holds with $\mathcal{A}$ replaced by $\mathcal{H}(\bx)$;
\paritem{(iii)} { for any $\eps>0$ there exists $\delta>0$ such that $\mathcal{H}(x) \subset \mathcal{H}(\bx) + \eps\ball$  whenever  $x \in \ball_\delta(\bx)$.}\\
 Then  $f+F$ is strongly subregular at $\bx$ for $\by$ {with modulus satisfying \eqref{smse} where $\mathcal{A}$ is replaced by $\mathcal{H}(\bx)$. }
\end{theorem}

\begin{proof}
{Let $m$ and $H_A$ be as in Theorem~\ref{thmSSSR} (ii) with $\mathcal{A}$ replaced by $\mathcal{H}(\bx)$. Then
there exists   $\gamma>0$ satisfying
\begin{equation} \label{eqmukappas}
 \big(c + \chi (\mathcal{H}(\bx))+2\gamma\big)  m   < 1 - \gamma m.
\end{equation}
By (iii), we may make $r$ smaller  if necessary to have
\begin{equation} \label{has}
 \mathcal{H} (u) \subset  \mathcal{H} (\bx) + \gamma \ball
 \quad \mbox{for each} \quad   u \in \ball_r(\bx).
\end{equation}
From  the definition of measure of non-compactness, there is a finite set $\mathcal{B} \subset  \mathcal{H}(\bar{x})$ such that
$$
 \mathcal{H} (\bx) \subset \mathcal{B} + \big(\chi (\mathcal{H}(\bx)) + \gamma \big) \ball.
$$
Hence, from  \eqref{has}, for any $u \in \ball_r(\bx)$ we get
$$
  \mathcal{H} (u) \subset   \mathcal{B} + \big(\chi (\mathcal{H}(\bx)) + \gamma \big) \ball +  \gamma  \ball = \mathcal{B} + (\chi (\mathcal{H}(\bx)) + 2\gamma) \ball;$$
that is, $$   \mathcal{H} (\ball_r(\bx))     \subset \mathcal{B}  +  (\chi (\mathcal{H}(\bx)) + 2\gamma)\ball.$$
This shows  that the measure of non-compactness of the set $\mathcal{A}:= \mathcal{H} (\ball_r(\bx))$ is not greater than $\chi (\mathcal{H}(\bx)) + 2\gamma$.
Since $h(u) \in \mathcal{H}(u) \subset \mathcal{A}$ for each $u \in \ball_r(\bx)$ the assumption (i) of Theorem~\ref{thmSSSR} holds. By \eqref{has} we have $\mathcal{A} \subset   \mathcal{H} (\bx) + \gamma \ball$. We will now prove that
\be\label{m'}
 m': =\sup\limits_{A \in \mathcal{A}} \subreg (H_A; \bx \for \by) \leq \frac{m}{1 - \gamma m}.
\ee
Choose  any $A \in \mathcal{A}$. Find $\bar{A} \in \mathcal{H} (\bx)$ such that $\|A - \bar{A}\| \leq \gamma$. Note that, by \eqref{eqmukappas}, we have $\gamma m < 1$. Inasmuch as
$ H_{A} = H_{\bar{A}} + (A - \bar{A}) (\cdot - \bar{x})$, Corollary~\ref{prop1} implies that
$\subreg (H_A; \bx \for \by) \leq m/(1-\gamma m)$. Since  $A \in \mathcal{A}$ was arbitrarily chosen in $\mathcal{A}$ we get \eqref{m'}.

Remembering \eqref{eqmukappas}, we have that $(c+ \chi(\mathcal{A})) m' < 1$; that is, the assumptions in (ii) of Theorem~\ref{thmSSSR} hold with $m$ replaced by $m'$.
Then  $f+F$ is strongly subregular at $\bx$ for $\by$ with modulus not greater than $ m'/(1-  (c+\chi(\mathcal{A})) m')$. This finishes the proof of \eqref{smse} with $\mathcal{A}:=\mathcal{H}(\bx)$}, because $\gamma > 0$ can be chosen arbitrarily close to $0$, which means that $\chi(\mathcal{A})$ and $m'$ can be made arbitrarily close to $\chi(\mathcal{H}(\bx))$ and $m$, respectively.
\end{proof}

Recall that a function $f:\reals^n \to \reals^m $ is said to be
semismooth at {$\bx \in \reals^n$}
when it is Lipschitz continuous
around $\bx$,  directionally differentiable  in every direction, and for every $c > 0$
there exists $r>0$ such that
$$
 \|f(u) - f(\bx) - A(u-\bx) \|\leq c \|u-\bx \|  \text{ for every} u \in \ball_r(\bx) \text{and every }A \in \partial_C f(u).
$$
If $f$ is semismooth at $\bx$ {then for any $c>0$ there is $r>0$ such that inequality \eqref{wsemi} is satisfied with $h$} being any selection  {of} $\partial_B f$;
thus   Theorem~\ref{thmSSSR1} is a subregularity version of a statement in \cite{Gowda}. It also yields
  a version of Corollary~\ref{c1} for Bouligand's limiting Jacobian {which is known to be outer semicontinuous (at any point $\bx \in \reals^n$) \cite[Proposition 7.4.11]{FP}, that is, condition (iii) in Theorem~\ref{thmSSSR1} holds.} 

\begin{corollary} \label{c2}
Let $(\bx,\by) \in \mathbb{R}^n \times {\mathbb{R}^m}$, $f:\mathbb{R}^n \to {\mathbb{R}^m}$ and  $F:\mathbb{R}^n \tto{\mathbb{R}^m}$ be such that $\by \in f(\bx) + F(\bx)$.
Suppose that $f$ is  Lipschitz continuous around
$\bx$
and that for every $c>0$ there exists
$r > 0$ along with a selection $h$ for $\partial_B f$ such that
\begin{equation} \label{wsemi1}
\|f(u)-f(\bx) - h(u)(u-\bx)\|\leq c \|u-\bx\|  \quad \mbox{whenever} \quad u \in \ball_r(\bx).
\end{equation}
Assume that, for each $A \in \partial_B f (\bx)$, the mapping $H_A$ defined in \eqref{HA} is strongly subregular at $\bx$ for $\by$.   Then $f+F$
 is strongly subregular at $\bx$ for $\by$; moreover,
 $$
 \subreg (f+F; \bx \for \by) \leq \sup\limits_{A \in \partial_B f (\bx)} \subreg (H_A; \bx \for \by).
 $$
\end{corollary}

Finally, we consider a derivative-type approximation of the function  $f$ by a positively  homogeneous set-valued mapping.
\begin{theorem} \label{thmSSSR2}
Let $X$ and $Y$ be Banach spaces and consider a function  $f:X\to Y$, a set-valued mapping  $F:X \tto Y$ and a point
 $(\bx,\by) \in X \times Y$ such that $\by \in f(\bx) + F(\bx)$. Suppose that there exist a positively homogeneous  mapping $G:X \tto Y$ and a constant  $c> 0$ such that
 \paritem{(i)} there exists a constant $r>0$ such that
 \be\label{eqpreder}
   f(u) - f(\bx) \in G(u - \bx) + c \|u - \bx\| \ball \quad \mbox{for every} \quad u\in \ball_r(\bx);
\ee
 \paritem{(ii)}  the mapping  $H:=f(\bx) + G(\cdot-\bx) + F$
is strongly subregular at $\bx$ for $\by$ with $\subreg (H; \bx \for {\by})   < 1/c$.\\
 Then  $f+F$ is strongly subregular at $\bx$ for $\by$; moreover
 \be\label{smse2}\subreg (f+F; \bx \for {\by}) \leq \frac{\subreg (H; \bx \for {\by}) }{1-  c \cdot \subreg (H; \bx \for {\by}) }.\ee
\end{theorem}
\begin{proof} {Let $\kappa >  \subreg (H; \bx \for {\by})$ be such that $c \kappa < 1$. Shrink $r$, if necessary, to have
$$
 \|x-\bx\| \leq \kappa d(\by, H(x)) \quad \mbox{for each} \quad x\in \ball_r(\bx).
$$
Choose any $x\in \ball_r(\bx)$ and then an arbitrary $y \in F(x)$. By \eqref{eqpreder} we find $w \in c \|x - \bx\| \ball$ such that  $f(x)-f(\bx)- w \in G(x - \bx)$.
Then $f(x)- w + y \in f(\bx) + G(x - \bx) + F(x) = H(x)$ and we have
\begin{eqnarray*}
\|x-\bx\| & \leq & \kappa d(\by, H(x)) \leq \kappa \|\by - f(x) + w - y\| \leq \kappa  \|(\by - f(x)) -  y\| + \kappa \|w\| \\
& \leq  &\kappa  \|(\by - f(x)) -  y\|  + \kappa c \|x - \bx\|.
\end{eqnarray*}
Therefore $(1 - c \kappa)\|x-\bx\| \leq \kappa \|(\by - f(x)) -  y)\|$ for any $y \in F(x)$. Thus, we have
$$
 \|x - \bx\| \leq \frac{\kappa}{1-c\kappa} d(\by-f(x),F(x)) = \frac{\kappa}{1-c\kappa} d\big(\by, (f+F)(x) \big).
$$
Noting that  $x$ was arbitrarily chosen in  $\ball_r(\bx)$ and $\kappa$ can be chosen arbitrarily close to $\subreg (H; \bx \for {\by})$,  the proof is complete.}
\end{proof}

Taking $F \equiv 0$ the above proof gives  a  direct proof of \cite[{Theorem 4.2}]{u}. We show next that  Theorem~ \ref{thmSSSR2}  implies Theorem~\ref{thmSSSR}.

\begin{remark} \label{remark}\rm
 Let $f$, $F$, $(\bx, \by)$, $\mathcal{A}$, $c$, $m$ and $r$ be as in Theorem~\ref{thmSSSR}. Define $G: X \tto Y$ by $G(u):=\{Au \mset A \in \mathcal{A}\}$, $u \in X$. Then the condition (i) in Theorem~\ref{thmSSSR} implies (i) in Theorem~\ref{thmSSSR2}.
The mapping $H$ from Theorem~\ref{thmSSSR2} (ii) has  $\subreg (H; \bx \for {\by}) \leq m/(1-m\chi(A)) =:m'$. Indeed, in the proof of \eqref{bbbs} we showed that for any  $\kappa > m$ and any $\gamma > 0$ sufficiently close to $m$ and $0$, respectively, there exists  $a \in (0,r]$  such that
$$
   \| x - \bx\| \leq \frac{\kappa}{1 - \kappa (\chi (\mathcal{A}) + \gamma)} d\big({\by}, f(\bx) + A(x-\bx) + F(x)\big) \quad \mbox{whenever} \quad x \in \ball_a(\bx)
   \ \mbox{ and } \ A\in \mathcal{A}.
$$
Fix any  $x \in \ball_a(\bx)$, and then pick arbitrary $v \in H(x)$ (if  any). The very definition of the mapping $H$ implies that there is $A \in \mathcal{A}$ such that $v \in  f(\bx) + A(x-\bx) + F(x)$. Then
$$
\| x - \bx\| \leq \frac{\kappa}{1 - \kappa (\chi (\mathcal{A}) + \gamma)} d\big({\by}, f(\bx) + A(x-\bx) + F(x)\big) \leq \frac{\kappa}{1 - \kappa (\chi (\mathcal{A}) + \gamma)} \|\by - v\|.
$$
Taking into account that $v$ is a fixed element of $H(x)$, and the constants $\kappa$ and $\gamma$ can be arbitrarily close to $m$ and $0$,
respectively,
we obtain the desired estimate for the subregularity modulus of $H$. Inequality \eqref{m} implies that
$m' c < 1$. Therefore condition (ii) in Theorem~\ref{thmSSSR2}  holds. Hence $f+F$ is strongly subregular at $\bx$ for $\by$ and
$$
\subreg (f+F; \bx \for \by) \leq \frac{m' }{1-  c m'} = \frac{m}{1-(c + \chi(\mathcal{A})) m}.
$$
\end{remark}

 A result analogous to Corollary~\ref{c1}  for strong regularity was stated  in
\cite{I}; a complete proof  extended to Banach spaces is given  in \cite{NR}. In a
more recent paper \cite{NG} a nonsmooth version of the Lyusternik-Graves theorem for metric regularity  is obtained. We note that the proofs in \cite{NR} and \cite{NG} are much more involved than the proofs of Theorems~\ref{thmSSSR} and
\ref{thmSSSR1} and use other conditions, {for example},
  convexity of the set $\mathcal{A}$ of derivative approximations.

\section{Strong $q$-subregularity}\label{Secq}

We  consider in this section an extension of the strong  metric subregularity,  the so-called strong  metric $q$-subregularity, defined as follows. For a positive scalar $q$, a mapping $F:X \tto Y$ acting between metric spaces $X$ and $Y$ is said to be strongly $q$-subregular at $\bx$ for $\by$ when $(\bx,\by) \in \mbox{gph} F$ and there exist a constant $\kappa\geq 0$ and  a neighborhood $U$ of $\bx$ such that
$$
      \rho(x,\bx) \leq \kappa d(\by, F(x))^q
      \; \text{for all} x \in U.
$$
The (usual) strong subregularity is obtained for $q=1$.

 Observe  that for $q\neq 1$  this property is not stable under linearization, in the sense of Proposition~\ref{propo}. As a counterexample take $F(x) = x^3$ with $\bx=0$.  However, if we consider perturbations by a function which is calm of order $1/q$,
then a simple modification of  the proof of Theorem~\ref{main} gives us perturbation stability. {Given $\gamma>0$, a function $g:X\to Y$ is said to be $\gamma$-calm  at $\bx \in\dom g$ with the constant $\mu \ge 0$ provided that there is a neighborhood $U$ of $\bx$ such that
$$
 \rho(g(x), g(\bx)) \leq \mu \rho(x, \bx)^{\gamma} \quad \mbox{for each} \quad x \in U \cap\dom g.
$$
The precise result is as follows:

\begin{theorem}\label{mainq}
Let $X$ be  a metric space and $Y$ be a linear metric space with shift invariant metric.
Let $a\in (0,1]$, $q>0$, and $\gamma \in [1/q,+\infty)$, and let  $\kappa$ and  $\mu$ be positive constants such that $\kappa \mu^q < 1$. Suppose that a mapping
$G:X \tto Y$ is strongly $q$-subregular at $\bx$ for $\by$ with   constant $\kappa$ and  neighborhood $\ball_a(\bx)$. Also, consider  a function $g:X\to Y$ which  is $\gamma$-calm  at $\bx$ with   constant $\mu$ and  neighborhood $\ball_a(\bx)$.
Then $g+G$ is  strongly $q$-subregular at $\bx$ for $\by+g(\bx)$ with  constant  $\kappa/(1-\kappa^{\frac{1}{q}}\mu)^q$ and  neighborhood $\ball_a(\bx)$.
\end{theorem}

\begin{proof} The proof repeats that of Theorem~\ref{main} with some
 adjustments of the
exponents.
By assumption, we have
$$
      \rho(x,\bx)^{\frac{1}{q}} \leq \kappa^{\frac{1}{q}} d(\by, G(x)) \text{ and } \rho(g(x), g(\bx))\leq \mu\rho(x,\bx)^\gamma
      \; \text{for all} x \in \ball_a(\bx) \cap \dom g.
$$
Observe that $\dom(g+G)=\dom g\cap\dom G$.
Take any $x \in \ball_a(\bx)\cap\dom g$. If $G(x)$ is empty we are done. If $G(x) \neq \emptyset$ then
\bas
\rho(x,\bx)^{\frac{1}{q}} & \leq & \kappa^{\frac{1}{q}} d(\by, G(x)) \leq  \kappa^{\frac{1}{q}} d(\by+g(\bx)-g(x), G(x)) + \kappa^{\frac{1}{q}} \rho(g(x), g(\bx)) \\
& \leq & \kappa^{\frac{1}{q}} d(\by+g(\bx), (g+G)(x)) + \kappa^{\frac{1}{q}} \mu\rho(x,\bx)^{\gamma}.
\eas
Since {$a \le 1$} and $\gamma \in [1/q,+\infty)$ we have $\rho(x,\bx)^{\gamma} \leq \rho(x,\bx)^{\frac{1}{q}}$.  Taking into account that $\kappa^{\frac{1}{q}} \mu < 1$, we obtain
$$\rho(x,\bx)^{\frac{1}{q}} \leq \frac{\kappa^{\frac{1}{q}}}{1-\kappa^{\frac{1}{q}}\mu}d(\by+g(\bx), (g+G)(x))$$
and the proof is complete.
\end{proof}

As in the standard case with $q=1$, when $X$ and $Y$ are Banach spaces and the perturbation is represented by a  Fr\'echet differentiable function, we can say more about perturbation stability.
\begin{theorem} \label{thmFD}
Let $X$ and $Y$ are Banach spaces
and let  $q\geq 1$ {and $(\bx,\by) \in X \times Y$. Consider} a function $f:X\to Y$ which is Fr\'echet differentiable at $\bx$
and a set-valued mapping $F:X\tto Y$ such that $\by \in f(\bx)+F(\bx)$. Then the mapping $f+F$ is strongly $q$-subregular at $\bx$ for $\by$ if and only if
the mapping $H:=f(\bx) + Df(\bx)(\cdot-\bx) +F$ has the same property.\\
\indent Assume, in addition, that  $f$ is Fr\'echet differentiable around $\bx$ and $Df$ is continuous at $\bx$. Then $f+F$ is strongly $q$-subregular
at $\bx$ for $\by$ if and only if there are $\lambda>0$ and $a>0$ such that for any $u \in \ball_a(\bx)$ the mapping $H_u:=f(\bx) + Df(u)(\cdot-\bx) +F$ is strongly $q$-subregular
at $\bx$ for $\by$ with constant $\lambda$ and neighborhood $\ball_a(\bx)$.
 \end{theorem}

\begin{proof}
The Fr\'echet differentiability of $f$ {means that} the function  $g:=f(\bx) + Df(\bx)(\cdot-\bx) - f$ {has}
$\clm(g;\bx) = 0$. Let $\kappa > 0$ be such that $f+F$ is strongly $q$-subregular at $\bx$ for $\by$ with constant $\kappa$.  Clearly,  there are  $\mu > 0$ and $a\in (0,1]$  such that $G:=f+F$ and $g$ satisfy the assumptions of  Theorem~\ref{mainq} with $\gamma=1$;  hence $H=G+g$  is strongly $q$-subregular at $\bx$ for $\by$. To prove the opposite implication, use $H$ and $-g$ as $G$ and $g$, respectively.

Now suppose that $f$ is continuously differentiable at $\bx$. Let $\kappa>0$ and $a\in(0,1)$ be such that the mapping $G:=f+F$ is strongly $q$-subregular at $\bx$ for $\by$ with constant $\kappa$ and neighborhood $\ball_a(\bx)$. Let $\mu > 0$ be such that $\kappa \mu^q < 1$. Using standard calculus and making $a$ smaller, if necessary, we have that
$$
\|f(x) - f(\bx) - Df(u)(x - \bx)\| \leq \mu \|x - \bx \| \quad \mbox{whenever} \quad x,u \in \ball_a(\bx).
$$
Fix any $u \in \ball_a(\bx)$. Then $g_u:=f(\bx) + Df(u)(\cdot-\bx) - f$ is calm at $\bx$  with a constant $\mu$ and a  neighborhood $\ball_a(\bx)$; moreover $g(\bx)=0$. Applying Theorem~\ref{mainq} with $\gamma=1$, we get that $H_u = G + g_u $ is strongly $q$-subregular at $\bx$ for $\by$ with a constant $\lambda := \kappa/(1-\kappa^{\frac{1}{q}}\mu)^q$, which is independent of $u$. The opposite direction follows from the first part of the statement.
\end{proof}

We end this section with some comments regarding the recent paper \cite{MO}. Taking $\gamma=1$ and $q\geq 1$ in Theorem~\ref{mainq}, one obtains \cite[Theorem 4.1]{MO} where the authors use the stronger  assumption  that the single-valued perturbation is Lipschitz continuous around $\bx$. The first part of Theorem~\ref{thmFD} slightly improves \cite[Corollary 4.2]{MO} where  strict differentiability of the single-valued part is assumed, while the second
 echoes  \cite[Theorem 4.4]{MO}.

\section{ Conditions involving generalized derivatives }\label{gd}

In this section $X$ and $Y$ are Banach spaces and  $X^*$ and $Y^*$ are their duals, respectively.
It follows directly from the definition
that a mapping $F:X \tto Y$ with $(\bx, \by) \in \gph F$ is strongly subregular at $\bx$ for $\by$ if and only if
its \emph{steepest displacement rate}  at $\bx$ for $\by$ defined as
 \begin{equation}\label{P5.1.1-1}
  |F|^{\downarrow}(\bx \for \by) := \liminf_{x\to\bx} \frac{d(\by,F(x))}{\|x - \bx\|}
\end{equation}
is positive (with the convention that the limit in \eqref{P5.1.1-1} is $+\infty$ when $\bx$ is an isolated point in $\dom F$). This notion was introduced by A. Uderzo in \cite{u}. It is elementary to check (see \cite[{Proposition 2.1}]{u}) that
\begin{equation} \label{eqsubdisp}
  |F|^{\downarrow}(\bx \for \by)  \cdot \subreg(F;\bx\for\by) = 1,
\end{equation}
where we set $0 \cdot (+ \infty) =  (+ \infty) \cdot 0  = 1$.  Thus, if $F$ is strongly subregular at $\bx$ for $\by$ with a constant $\kappa > 0$ then we have $|F|^{\downarrow}(\bx \for \by) \geq \kappa^{-1}$. Conversely, if $|F|^{\downarrow}(\bx \for \by) > \kappa^{-1}$ for some $\kappa >0$ then $F$ is strongly subregular at $\bx$ for $\by$ with the constant $\kappa$.

When $\bx$ is not an isolated point in $F^{-1}(\by)$, then $|F|^{\downarrow}(\bx \for \by)=0$.
Otherwise, the steepest displacement rate \eqref{P5.1.1-1} coincides with the subregularity constant
\begin{equation*}
^sr[F](\bx,\by):= \liminf_{\substack{x\to\bx\\x\notin F^{-1}(\by)}} \frac{d(\by,F(x))}{d(x,F^{-1}(\by))}
\end{equation*}
extensively used in \cite{Kru15} when characterizing metric subregularity.

First, we focus on conditions based on tangential approximation of the graph of the mapping in question. Let $\Omega$ be a set in  $X$ and let $\bx \in \Omega$. The \emph{Bouligand-Severi tangent cone} to $\Omega$ at $\bx$, denoted by $T_\Omega(\bx)$,  is the set of all $w \in X$ such that there are
sequences $\{w_k\}$ in $X$ and $\{t_k\}$ in $(0, +\infty)$ converging to $w$ and $0$, respectively, such that $\bx + t_k w_k \in \Omega$ for each $k \in \N$.
For a mapping  $F : X \tto Y$  with $(\bx,\by) \in \gph F$,  the \emph{graphical derivative mapping}  of $F$ at  $(\bx,\by)$ is defined as
$$
X \ni u \mapsto  {D}F(\bx\for\by)(u):= \{v \in  Y \mset \ (u,v) \in  T_{\gph F} (\bx,\by) \}.
$$
The following is a  generalization of  \cite[Theorem 4E.1]{book} which goes back to Rockafellar  \cite{R89}:

\begin{theorem} \label{thm4E1}
Consider a mapping $F:X \tto Y$ with $(\bx, \by)\in\gph F$.   Then
\begin{equation} \label{ssregDF}
         \| DF(\bx\for\by)^{-1} \| ^\plus \leq \subreg (F;\bx\for\by).
\end{equation}
If, in addition,  the dimension of $X$ is finite,  then
$$
   \subreg (F;\bx\for\by) < +\infty  \quad \Longleftrightarrow \quad     \| DF(\bx\for\by)^{-1} \| ^\plus <+\infty  \quad \Longleftrightarrow \quad   DF(\bx\for\by)^{-1}(0)=\{0\};
$$
 that is, $F$ is strongly subregular at $\bx$ for
$\by$ if and only if $\| DF(\bx\for\by)^{-1} \| ^\plus$ is finite.
\newline\indent Moreover,  if both $X$ and  $Y$ are finite-dimensional, then \eqref{ssregDF} holds as  equality.
\end{theorem}
\begin{proof}
For  the first part of the claim, note that if the right-hand side of  \eqref{ssregDF} is infinite then we are done. If not, pick  $\kappa > \subreg (F;\bx\for\by)$ and then  $a > 0$ such that
$$
 \|x - \bx \| \leq \kappa d(\by, F(x)) \quad \mbox{for each} \quad x \in \ball_a(\bx).
$$
Fix an arbitrary $(u,v)  \in \gph DF(\bx\for\by) = T_{\gph F}(\bx,\by)$. Then there exist sequences $\{u_k\}$ in $X$ and
$\{v_k\}$ in $Y$, as well as  $\{t_k\}$ in $(0,1)$, converging to $u$, $v$, and $0$, respectively, such that
$ \by +t_kv_k \in F(\bx + t_ku_k)$ for each $k \in \N$. For $k$ sufficiently large we have $x_k:=\bx + t_k u_k  \in \ball_a(\bx)$ and hence
$$
 t_k \|u_k\| = \|x_k - \bx\| \leq \kappa \|\by + t_k v_k - \by\| = t_k \kappa \|v_k\|.
$$
Consequently, $\|u \| \leq \kappa \|v\|$ for each $(u,v)  \in \gph DF(\bx\for\by)$. Thus $\| DF(\bx\for\by)^{-1} \| ^\plus \leq \kappa$.  Letting $\kappa \downarrow \subreg (F;\bx\for\by)$ we get \eqref{ssregDF}.

Now, let   $X$ be finite-dimensional.  By \cite[Proposition 5A.7]{book} we know that $\| DF(\bx\for\by)^{-1} \| ^\plus$ is finite if and only if $DF(\bx\for\by)^{-1}(0)=\{0\}$ . In view of \eqref{ssregDF}, it is sufficient to prove the $\Longleftarrow$ part in the first equivalence. Let $\| DF(\bx\for\by)^{-1} \| ^\plus$ be finite. Suppose on the contrary that $F$ is not strongly subregular at $\bx$ for $\by$. Then there is  a sequence $\{(x_k,y_k)\}$ in $\gph F$ converging to $(\bx, \by)$ such that
 $$
 \|x_k - \bx\| > k \|y_k - \by\| \quad \mbox{for each} \quad k \in \N.
$$
Let $t_k:= \|x_k - \bx\|$, $u_k:=(x_k - \bx)/t_k$, and $v_k := (y_k - \by)/t_k$, $k \in \N$. By the above inequality, $t_k \downarrow 0$ and $v_k \to 0$ as $k \to + \infty$. Since $X$ is finite-dimensional, we can assume that  $\{u_k\}$ converges  to  some $u \in X$ with $\|u\|=1$. Noting that
\begin{equation} \label{eqbybx}
 \by + t_k v_k = y_k \in F(x_k) = F(\bx + t_k u_k) \quad \mbox{for each} \quad k \in \N,
\end{equation}
we get that $0 \in DF(\bx\for\by)(u)$ for $u \neq 0$, that is, $\| DF(\bx\for\by)^{-1} \| ^\plus = +\infty$, a contradiction.

Let   $Y$ be  finite-dimensional as well. Suppose that \eqref{ssregDF} is strict; then there is a (positive) constant $\kappa$ such that  $ \| DF(\bx\for\by)^{-1} \| ^\plus < \kappa < \subreg (F;\bx\for\by)$. Find a sequence $\{(x_k,y_k)\}$ in $\gph F$ converging to $(\bx, \by)$ such that
\begin{equation} \label{xkkappa}
 \|x_k - \bx\| > \kappa \|y_k - \by\| \quad \mbox{for each} \quad k \in \N.
\end{equation}
Let $\{t_k\}$, $\{u_k\}$, and $\{v_k\}$ be defined as in the previous paragraph. For each $k \in \N$, we have $t_k > 0$, $\|u_k\|=1$, and $ v_k \in \kappa^{-1}\ball$. Also $ t_k \downarrow 0$  as $k \to +\infty$.  Since both $X$ and $Y$ are finite-dimensional, we can assume that  $\{u_k\}$ converges  to  some $u \in X$ with $\|u\|=1$ and that $\{v_k\}$ converges to some $v \in \kappa^{-1} \ball$. By \eqref{eqbybx} we conclude that $v \in DF(\bx\for\by)(u)$.
Dividing  \eqref{xkkappa} by $t_k$ and taking the limit as $k \to + \infty$  we get $ \| u \| = 1 \geq \kappa \| v\|$.
Hence  $\| DF(\bx\for\by)^{-1} \| ^\plus \geq \kappa$, a contradiction.
\end{proof}

We will now consider dual space conditions for strong subregularity. Unless clearly indicated otherwise, we equip $X \times Y$ with the product (box) topology.
Given a set $\Omega \subset X$ and a point $\bx \in \Omega$,  the \emph{Fr\'echet normal cone} to $\Omega$ at  $\bx$, denoted by  $\widehat{N}_{\Omega}(\bx)$, is the set of all $x^* \in X^*$ such that for every $\varepsilon > 0$ there exits $\delta > 0$ such that
$$
  \langle x^*, x - \bx \rangle \leq \varepsilon \|x - \bx \| \quad \mbox{whenever} \quad x \in \Omega \cap \ball_{\delta}(\bx).
$$
For a mapping  $F : X \tto Y$  with $(\bx,\by) \in \gph F$,
the  \emph{Fr\'echet coderivative}  of $F$ at $(\bx,\by)$ acts from $  Y^* $ to the subsets of $X^*$ and is defined as
$$
Y^* \ni y^* \mapsto  \widehat{D}^*F(\bx\for\by)(y^*):= \left\{x^* \in  X^* \mset \ (x^*,-y^*) \in  \widehat{N}_{\gph F} (\bx,\by) \right\}.
$$

We give next  coderivative conditions for strong subregularity:

\begin{theorem} \label{thmCoder}
Consider a mapping $F:X \tto Y$ with $(\bx, \by)\in\gph F$.  If $X$ is finite-dimensional, then
\begin{equation} \label{DFDN}
    \subreg (F;\bx\for\by)  \leq \|\widehat{D}^*F(\bx\for\by)^{-1}\|^-.
\end{equation}
If, in addition, $\gph F$ is locally convex at $(\bx,\by)$, meaning that $\gph F \cap W$ is convex for some neighborhood $W$ of $(\bx,\by)$ in $X \times Y$, then \eqref{DFDN} becomes an equality.
\end{theorem}

\begin{proof}
If either the right-hand side of \eqref{DFDN} is infinite or $\bx$ is an isolated point of $\dom F$ (implying that the left-hand side of \eqref{DFDN} is zero) then we are done. Suppose that this is not the case, and fix any $\kappa > \|\widehat{D}^*F(\bx\for\by)^{-1}\|^-$.

First, we show that
\begin{equation} \label{eqliminf}
  \liminf_{x\to\bx, x \neq \bx } \frac{d(\by,F(x))-\langle x^*,x- \bx \rangle}{\|x-\bx\|}\ge0 \quad \mbox{for all} \quad  x^* \in\kappa^{-1}\ball.
\end{equation}
To obtain \eqref{eqliminf}, it is sufficient to show that, given $x^* \in X^*$ with $\|x^*\| \leq 1$,  for each $\gamma \in (0, 1)$ there is  a constant $\delta=\delta(x^*, \gamma) > 0$ such that
\begin{equation} \label{eqSubDif}
  \langle x^*, x - \bx \rangle \leq  \kappa d(\by, F(x)) + \gamma \|x - \bx\| \quad \mbox{whenever} \quad x\in\ball_\delta(\bx).
\end{equation}
Assume on the contrary that there are $x^* \in X^*$ with $\|x^*\| \leq 1$ and $\gamma \in (0,1)$ along with a sequence $\{x_k\}$ converging to $\bx$ such that
$$
   \langle x^*, x_k - \bx \rangle >  \kappa d(\by, F(x_k)) + \gamma \|x_k - \bx\| \quad \mbox{for each} \quad k \in \N.
$$
For each  $k \in \N$, choose a point $y_k\in F(x_k)$ such that
\begin{equation} \label{eqxsbx}
   \langle x^*, x_k - \bx \rangle >  \kappa \|y_k - \by \| + \gamma \|x_k - \bx\|;
\end{equation}
this means  in particular that
\begin{equation} \label{eqykby}
   \kappa \|y_k - \by \| < (1 - \gamma) \|x_k - \bx\|.
\end{equation}
The choice of $\kappa$  implies that there is
$y^* \in \widehat{D}^*F(\bx\for\by)^{-1}(x^*)$ with $\|y^*\| \leq \kappa$. Hence, we have
$(x^*, - y^*) \in \widehat{N}_{\gph F} (\bx, \by)$. Let
$$
 \varepsilon := \gamma \min \{1,\kappa/(1 - \gamma)\}.
$$
Observe  that \eqref{eqykby} implies that $\{y_k\}$ converges to $\by$ and
$$
 \varepsilon \max\{\|x_k - \bx\|, \|y_k - \by\|\} \leq \gamma \|x_k - \bx\|
 \quad \mbox{whenever} \quad k \in \N.
$$
For each $k \in \N$, using \eqref{eqxsbx}, we obtain
\begin{eqnarray*}
  \langle x^*,x_k- \bx \rangle -\langle y^*,y_k-\by \rangle
 &\geq &  \langle x^*,x_k- \bx \rangle -\kappa\|y_k-\by\| \\
 &>& \gamma \|x_k - \bx\| \ge \varepsilon \max\{\|x_k - \bx\|, \|y_k - \by\|\}.
\end{eqnarray*}
Thus $(x^*, - y^*) \notin \widehat{N}_{\gph F} (\bx, \by)$, a contradiction. We proved that \eqref{eqSubDif} holds, and consequently so does \eqref{eqliminf}.

Second, we show that  \eqref{eqliminf} implies that $|F|^{\downarrow}(\bx \for \by) \geq 1/\kappa$.
Indeed, let $\{x_k\}$ be any sequence in $X\setminus \{ \bx \}$ converging to $\bx$ such that
$$
  \lim_{k \to + \infty} \frac{d(\by,F(x_k))}{\|x_k - \bx\|} = |F|^{\downarrow}(\bx \for \by).
$$
Let $u_k := (x_k - \bx)/\|x_k - \bx\|$, $k \in \N$. By Hahn-Banach theorem, for each $k \in \N$, there is $u^*_k \in X^*$ with $\|u^*_k\| = 1$ such that
$ \langle u^*_k, u_k \rangle = 1$. Going to subsequences, if necessary, we may assume that $\{u_k\}$ converges to some $u \in X$ with $\|u\| = 1$ and  that $\{u_k^*\}$ converges to some $u^* \in X^*$ with $\|u^*\| = 1$. Then
\begin{eqnarray*}
  0 \leq  |\langle u^*, u_k \rangle - 1 |  = |\langle u^*, u_k \rangle - \langle u^*_k, u_k \rangle | & \leq  &  \|u^* - u^*_k \|   \longrightarrow 0 \quad \mbox{as} \quad k \to + \infty.
\end{eqnarray*}
Let $x^* := u^*/\kappa$.
  Then \eqref{eqliminf} implies that
$$
 \lim_{k \to + \infty} \frac{d(\by,F(x_k))}{\|x_k - \bx\|} \geq \liminf_{k \to +\infty}   \frac{d(\by,F(x_k))-\langle x^*,x_k- \bx \rangle}{\|x_k-\bx\|} + \lim_{k \to + \infty} \langle x^*,u_k\rangle \geq 0 +  \frac{1}{\kappa} = \frac{1}{\kappa}.
$$
By \eqref{eqsubdisp}, we have  $ \subreg (F;\bx\for\by)  = 1/|F|^{\downarrow}(\bx \for \by) \leq \kappa$.  Letting $\kappa \downarrow \|\widehat{D}^*F(\bx\for\by)^{-1}\|^-$, we get
\eqref{DFDN}.

Suppose now  that $\gph F$ is locally convex at $(\bx,\by)$. We will show the  inequality opposite to \eqref{DFDN}.   Fix an arbitrary $\kappa > \subreg (F;\bx\for\by)$ (if any). Then  there is $\delta > 0$ such that $\Omega:=\gph F \cap (\ball_\delta(\bx) \times \ball_\delta(\by))$ is convex and
$$
   \|x - \bx\| \leq   \kappa d\big(\by,F(x)\cap \ball_\delta(\by)\big) \quad \mbox{for each} \quad x \in \ball_{\delta}(\bx).
$$
Clearly, in this case  $N_{\Omega}(\bx,\by) = \widehat{N}_{\gph F} (\bx,\by)$, where $N_{\Omega}$ is the usual normal cone to $\Omega$ at $(\bx, \by)$ in sense of convex analysis.  For any $x^*$ from the dual ball of $X$, we have
$$
  \langle x^*,x-\bx \rangle \leq  \kappa \|y-\by \| \quad \mbox{whenever} \quad (x,y)\in \Omega,
$$
that is, $(x^*,0)$ is a subgradient at $(\bx,\by)$ of the sum of two convex functions on $ \Omega$: the continuous function $\Omega \ni (x,y)\mapsto \kappa\|y-\by\|$ and the indicator function of the set $\Omega$, which is convex but not necessarily closed.
 Applying the convex sum rule \cite[Theorem 3.39]{JPP}, we get
$$\ball_{X^*} \times \{0\} \subset  N_{\Omega }(\bx,\by)+\{0\}\times (\kappa \ball_{Y^*}).$$
Hence for any  $x^* \in  X^*$ with $\|x^*\| \leq 1$ there is $y^* \in [ D^*F(\bx\for\by)]^{-1}(x^*)$ with $\|y^*\| \leq \kappa$.  Thus
$\|\widehat{D}^*F(\bx\for\by)^{-1}\|^- \leq \kappa$. Letting $\kappa \downarrow \subreg (F;\bx\for\by)$ we get the desired inequality.
\end{proof}

Note that inequality \eqref{DFDN} in Theorem~\ref{thmCoder} may be strict
rather often.
For instance, if the normal cone $\widehat{N}_{\gph F} (\bx,\by)$ is trivial, then $\|\widehat{D}^*F(\bx\for\by)^{-1}\|^-=+\infty$.
Take, {for example},
$F: \reals \tto \reals$ defined by $F(x) = \{x,-x\}$, $x \in \reals$. Then $\|DF(0\for 0 )^{-1}\|^\plus = \subreg (F; 0 \for 0) = 1$ while $\|\widehat{D}^*F(0\for0)^{-1}\|^- = +\infty$.  This particular example was also mentioned in the introduction to  illustrate the differences among the regularity properties for set-valued mappings.

Suppose that $X$ is finite-dimensional. Combining Theorem~\ref{thmCoder} and Theorem~\ref{thm4E1}, we get that for any $F: X \tto Y$ with $(\bx,\by) \in \gph F,$
$$
   \| DF(\bx\for\by)^{-1} \| ^\plus \leq \subreg (F;\bx\for\by) \leq \|\widehat{D}^*F(\bx\for\by)^{-1}\|^-.
$$
For any two positively homogeneous mappings $H_1$, $H_2 : Y \tto X$ such that $\gph H_1 \subset \gph H_2$ we have $\|H_2 \|^- \leq  \|H_1 \|^-$. Hence one
could
expect that taking a coderivative of $F$ at $(\bx,\by)$ based on a bigger normal cone than the Fr\'echet one we can achieve that its inner norm equals to
$\|DF(\bx\for \by )^{-1}\|^\plus$ and, therefore to  $\subreg (F; \bx \for \by)$.  In finite dimensions, a candidate for that to happen
could be
the \emph{limiting coderivative} ${D}^*F(\bx\for \by) :  \reals^m \tto \reals^n$ with values
$$
 {D}^*F(\bx\for\by)(q):= \{p \in  \reals^n \mset \ (p,-q) \in  N_{\gph F} (\bx,\by) \}, \quad  q \in \reals^m,
$$
where the \emph{limiting normal cone} $N_\Omega(\bz)$ to $\Omega \subset \reals^d$ at  $\bz \in \Omega$ is a collection of vectors
$w \in \reals^d$ such that there are sequences $\{w_k\}$ in $\reals^d$ and $\{z_k\}$ in $\Omega$ converging to $w$ and $\bz$, respectively, such that $w_k \in \widehat{N}_{\Omega}(z_k)$ for each $k \in \N$. However, the limiting coderivative cannot provide a criterion for strong subregularity, in general. As a counterexample, let $F :\reals \tto \reals $ be defined by $\gph F = \{ (1/k, 0): k \in \N \} \cup \{ (0,0)\}$.
Then $F$ is not strongly subregular at $0$ for $0$ and $\|DF(0\for 0 )^{-1}\|^\plus = \|\widehat{D}^*F(0\for0)^{-1}\|^- = \subreg (F; 0 \for 0)   = +\infty$, but $N_{\gph F}(0,0) = \reals^2$ which means that $\|{D}^*F(0\for0)^{-1}\|^-$ is finite.

Given $\varrho > 0$, we consider an equivalent norm in the product space $X \times Y$ defined by
$$
 \|(u,v)\|_\varrho := \max \{ \|u\|, \varrho \|v\|\}, \quad (u,v) \in X \times Y.
$$
Now we present a necessary and sufficient condition for strong subregularity similar to the statement by Fabian and Preiss \cite{fp}  guaranteeing that a set-valued mapping is open with a linear rate at a reference point. Note that this statement was  proved independently by Ioffe \cite{i2000} who showed that it implies  openness with a linear rate  around the reference point.

\begin{theorem} \label{thmCrit}
Consider a mapping $F:X \tto Y$ the graph of which is locally closed at $(\bx, \by)\in\gph F$.  Then $|F|^{\downarrow}(\bx \for \by)$ equals to the supremum
of  $\tau > 0$ for which there exists  $\varrho>0$ such that for any $(x,y)\in\gph F$ with $0< \|x - \bx\|<\varrho$ and $\|y-\by\|<\varrho$, one can find a point $(u,v)\in\gph F\setminus\{(x,y)\}$ satisfying
\begin{equation}\label{eqCrit}
 \|y - \by \|- \|v-\by\|>\tau \|(u,v) - (x,y)\|_\varrho.
\end{equation}
\end{theorem}

\begin{proof}
 Denote by $s$ the supremum from the statement and let $\ell :=|F|^{\downarrow}(\bx \for \by)$.

 First, we show that $\ell \leq s$. If $\ell =0$, the inequality holds  trivially. If not then fix any $\tau \in (0, \ell)$.
Find $\varrho\in (0,1/\tau)$ such that
\begin{equation}\label{3arb}
 \|y - \by\|> \tau \|x - \bx\| \quad \mbox{whenever} \quad x\in \ball_\varrho(\bar{x})\setminus\{\bx\} \quad  \mbox{and} \quad  y\in F(x).
\end{equation}
Fix an arbitrary $(x,y)\in\gph F$ with $0<\|x - \bx\|<\varrho$ and $\|y - \by\|<\varrho$.
Then $(u,v) := (\bx, \by)$ is distinct from $(x,y)$ and \eqref{3arb} implies that
$$
  \|y - \by \| >   \tau \|x - \bx\| = \tau \|u- x\| > 0.
$$
Hence $y \neq \by$. As $ \tau \varrho < 1$, we have  $\|y - \by \| > \tau \varrho  \|y - \by \| = \tau \varrho \|v - y\|$.   Noting that $\|y - \by \|- \|v-\by\| = \|y - \by \|$, we arrive at \eqref{eqCrit}.  Thus $ s\ge \tau $.
The claimed inequality follows after letting $\tau \uparrow\ell$.

To show that $\ell = s$, assume on the contrary that $\ell < s$. Choose $\delta\in (0,1)$ such that the set $M:=\gph F\cap(\ball_{\delta}(\bx) \times \ball_{\delta}(\by))$ is closed in $X \times Y$.
Fix any $\tau \in (\ell, s)$ and then pick $\tau' \in (\ell, \tau)$. Let $\varrho \in (0,\delta/2)$ be arbitrary, and set
\begin{gather}\label{3ch-1b}
\eta:=\min\{\varrho/2,\varrho\delta/2,\varrho/\tau,\delta/4\}.
\end{gather}
As $\tau' > \ell$, there is $z\in\ball_{\eta}(\bx)$ different from $\bx$ and $w\in F(z)$ such that
\begin{equation}\label{3ch0b}
\| w - \by \| <\tau' \|z - \bx\|.
\end{equation}
Consider a function $(u,v)\mapsto \|v - \by \|$ on a complete metric space $(M,\|\cdot \|_\varrho)$.
Applying to this function the Ekeland variational principle \cite[Theorem 7.1.2]{BL}
with
\begin{gather}\label{3ch1-b}
  \eps:= \tau ' \|z - \bx\| > 0  \quad \mbox{and} \quad  \lambda:=\frac{\tau'}{\tau}\|z - \bx\| > 0,
\end{gather}
we find a point $(x,y)\in M $ such that
\begin{gather}\label{3ch1b}
\|(x,y) - (z,w)\|_\varrho\le\lambda,\quad
\|y - \by\| \le \| w - \by\|,
\\\label{3ch2b}
 \| v - \by \| +(\eps/\lambda) \|(u,v) -(x,y) \|_\varrho\ge \| y - \by\| \quad \mbox{for all} \quad (u,v)\in M.
\end{gather}
 Using  \eqref{3ch-1b},  \eqref{3ch0b},  \eqref{3ch1-b}  and \eqref{3ch1b}  we have
\begin{eqnarray*}
 \|x - z \| &\le &\lambda<\|z - \bx\|,
\\
\|x - \bx\| &\ge & \|z - \bx\|-\|x - z\|  >0,
\\
\|x - \bx\| & \le & \|x - z\|+\|z - \bx\|< 2 \|z-\bx\|\le 2\eta
\le\min\{\varrho,\delta/2\},
\\
\|y - \by\| &\le & \|w - \by\|< \tau'  \|z - \bx\| < \tau \|z - \bx\| \le \tau \eta
\le\varrho< \delta/2.
\end{eqnarray*}
Thus we have $0< \| x - \bx\|<\varrho$ and $\|y - \by\|<\varrho$, and, as $\varrho < 1$,  also that
\begin{equation} \label{3lowb}
 \|(x,y) - (\bx,\by) \|_\varrho<\delta/2.
\end{equation}
Since \eqref{3ch1-b} means that $\varepsilon/\lambda = \tau$, from \eqref{3ch2b} we get
$$
 \|y - \by \| -\|v - \by\| \le \tau  \|(u,v) - (x,y)\|_\varrho \quad \mbox{for all} \quad (u,v)\in M.
$$
If $(u,v)\in \gph F \setminus M$, then, by \eqref{3lowb},
$$\|(u,v) - (x,y) \|_\varrho \geq \|(u,v) - (\bx,\by)\|_\varrho - \|(x,y) - (\bx,\by)\|_\varrho >\delta- \delta/2 = \delta/2,$$
which in combination with \eqref{3ch1b}, \eqref{3ch0b}, and \eqref{3ch-1b} implies that
\begin{eqnarray*}
 \|y -\by\| -\| v - \by\| &\le & \|y - \by\| \le \|w - \by\|< \tau \|z - \bx\| \le\tau \eta \le \tau \varrho\delta/2\\
&<&\tau \varrho \|(u,v)-(x,y)\|_\varrho  < \tau \|(u,v)-(x,y)\|_\varrho.
\end{eqnarray*}
Summarizing, we have shown that for every $\tau \in (\ell, s)$ and every $\varrho \in (0,\delta/2)$ there exists $(x,y)\in\gph F$ with $0< \|x - \bx\|<\varrho$ and $\|y-\by\|<\varrho$ such that no point $(u,v)\in \gph F $ can satisfy \eqref{eqCrit}. Hence $s$ cannot be strictly greater than $\ell$, a contradiction.
\end{proof}

We immediately get a statement characterizing strong subregularity  via
local and nonlocal \emph{slopes/rates of descent}.

\begin{corollary}\label{C2.1.17}
Consider a mapping $F:X \tto Y$ the graph of which is locally closed at $(\bx, \by)\in\gph F$. Then
$F$ is strongly subregular at $\bx$ for $\by$ if and only if
\begin{equation} \label{limcritSub}
\lim_{\varrho\downarrow0}
\inf_{\substack{(x,y)\in\gph F
\\
0<\|x - \bx\|<\varrho,\,\|y - \by\|<\varrho}}\,
\sup_{\substack{(u,v)\in\gph F\setminus\{(x,y)\}}}
\frac{\|y - \by\|-\|v - \by\|}{ \|(u,v) - (x,y)\|_\varrho}>0.
\end{equation}
Moreover, the limit in \eqref{limcritSub} coincides with $(\subreg(F;\bx\for\by))^{-1}$.
\end{corollary}

The limit \eqref{limcritSub} is taken in the product space $X\times Y$ and involves all points $(x,y)\in\gph F$ near $(\bx,\by)$ excluding those with $x=\bx$ (external points).
At every such point a kind of (nonlocal) descent rate is computed for the distance from $y$ to $\by$ and can be underestimated by the corresponding easier to compute infinitesimal quantities:
\begin{gather}\label{2.1.17}
\sup_{\substack{(u,v)\in\gph F\setminus\{(x,y)\}}}
\frac{\|y - \by\|-\|v - \by\|}{ \|(u,v) - (x,y)\|_\varrho}\ge \limsup_{\substack{(u,v)\to(x,y)\\(u,v)\in\gph F\setminus\{(x,y)\}}}
\frac{\|y - \by\|-\|v - \by\|}{ \|(u,v) - (x,y)\|_\varrho}.
\end{gather}
By analogy with the \emph{strong slope} by De Giorgi, Marino, and Tosques \cite{DMT},
the quantity on the right-hand side  of \eqref{2.1.17} can be interpreted as a kind of \emph{slope} of  $F$ at $(x,y)\in\gph F$ (cf. \cite{Kru15}).
It is easy to check that, when $\gph F$ is convex, \eqref{2.1.17} holds as equality.

\section{The Newton method}

We  study the Newton method for solving  the generalized equation
\be\label{geneq}
     \text{find} x \text{such that} f(x) + F(x) \ni 0,
\ee
 where both $X$ and $Y$ are Banach spaces,  $f:{X}\to {Y}$ is a function, and  $F:X\tto Y$ is a
 set-valued mapping. Provided that $f$ is
Fr\'echet differentiable,
the Newton  iteration applied to \eqref{geneq}  has the form
\be\label{New}
f(x_k)+ Df(x_k)(x_{k+1}-x_k) + F(x_{k+1})\ni 0, {\quad k=0,1,2, \dots, \quad x_0 \in X \mbox{ given}.}
\ee
In \cite[Chapter 6]{book} several results are presented  regarding the method \eqref{New} under (strong) metric (sub)regularity. In the following subsections we extend some of these results and add new ones.

\subsection{Convergence}\label{con}

The following theorem reveals the mode of convergence of the iteration \eqref{New} under strong subregularity of
the mapping in \eqref{geneq}. It improves \cite[Theorem 6E.2]{book}.

\begin{theorem}\label{N}   Suppose that
the function $f$ is   Fr\'echet differentiable around a solution $\bx$ of \eqref{geneq}  and the derivative mapping $Df$ is continuous at $\bx$. Also suppose that the mapping $f+F$  is strongly subregular at $\bx$ for $0$.
Then  there exists a neighborhood $O$ of $\bx$ such that if a  sequence $\{x_k\}$  is generated by the Newton method \eqref{New}
and  has a tail $\{x_k\}_{k\geq k_0}$
with
$x_k\in O$ for all $k\geq k_0$, then $\{x_k\}$ is
  superlinearly convergent to $\bx$.
\end{theorem}

\begin{proof}
The continuous differentiability of $f$ implies that for each $\mu > 0$ there is $\delta>0$ such that
\begin{equation} \label{fufv}
 \|f(v) - f(u) - Df(u)(v-u)\| \leq \mu \|u - v\| \quad \mbox{whenever} \quad u,v \in \ball_{\delta} (\bx).
\end{equation}
By the strong subregularity of $f+F$, there are positive constants $\kappa$ and $a$ such that
\begin{equation} \label{xbxd}
 \|x - \bx \| \leq \kappa d\big(0, f(x) + F(x)\big) \quad \mbox{for each} \quad x \in \ball_a(\bx).
\end{equation}
Let $\delta > 0$ be such that \eqref{fufv} holds with $\mu :=1/(3\kappa)$ and set $O=\ball_a(\bx)\cap\ball_\delta(\bx)$. Let $\{x_k\}$  be any sequence generated by the Newton method \eqref{New}
such that there is $k_0 \in \N$ such that  $x_k\in O$ for all $k \geq k_0$. For any $k \ge k_0$ we have $f(x_{k+1}) - f(x_k) - Df(x_k)(x_{k+1} - x_k) \in f(x_{k+1}) + F(x_{x+1})$ and thus
\begin{eqnarray*}
 \|x_{k+1} - \bx \| &\leq& \kappa d\big(0, f(x_{k+1}) + F(x_{k+1})\big) \leq \kappa \|f(x_{k+1}) - f(x_k) - Df(x_k)(x_{k+1} - x_k) \| \\
  &\leq& \frac{1}{3} \|x_{k+1} - x_k\| \leq  \frac{1}{3} \|x_{k+1} - \bx \| + \frac{1}{3} \|x_{k} - \bx \|.
\end{eqnarray*}
Therefore $\|x_{k+1} - \bx \| \leq 2^{-1} \|x_{k} - \bx \|$ for each $k \ge k_0$. Hence $x_k \to \bx$ as $k \to +\infty$. To see the rate of convergence, let $\varepsilon > 0$ be arbitrary. Find $r > 0$ such that $\ball_r(\bx) \subset O$ and \eqref{fufv} holds with $\mu= \varepsilon/(\kappa(1+\varepsilon))$ and $\delta=r$. Then there is $k_1 \in \N$ such that
$x_k \in \ball_r(\bx)$ whenever $k > k_1$. As above, for such an index $k$, we get
\begin{eqnarray*}
 \|x_{k+1} - \bx \| &\leq& \kappa d\big(0, f(x_{k+1}) + F(x_{k+1})\big) \leq \kappa \|f(x_{k+1}) - f(x_k) - Df(x_k)(x_{k+1} - x_k) \| \\
  &\leq& \frac{\varepsilon}{1 + \varepsilon} \|x_{k+1} - x_k\| \leq  \frac{\varepsilon}{1 + \varepsilon} \|x_{k+1} - \bx \| + \frac{\varepsilon}{1 + \varepsilon} \|x_{k} - \bx \|.
\end{eqnarray*}
Therefore for any $k > k_1$ we have $  \|x_{k+1} - \bx \| \leq \varepsilon \|x_{k} - \bx \|$. Hence $x_k \to \bx$ superlinearly.
\end{proof}

Clearly, the theorem above can be equivalently stated with the assumption that the entire sequence $\{x_k\}$ belongs to $O$; the statement we choose
adds some information which can be meaningful numerically.

Our next theorem extends the result just presented to the case of  strong $q$-regularity.

\begin{theorem}   \label{rem}  Assume that $Df$ is H\"older continuous around $\bx$ with an exponent $\alpha \in (0,1]$ and that $f + F$ is strongly   $q$-subregular at $\bx$ for $\by$ with $q \geq 1$. Then  there exists a neighborhood $O$ of $\bx$ such that if a  sequence $\{x_k\}$  is generated by the Newton method \eqref{New}
and  has a tail $\{x_k\}_{k\geq k_0}$
with
$x_k\in O$ for all $k\geq k_0$, then $\{x_k\}$ is
   convergent to $\bx$  with convergence rate  {$q(1+\alpha)$.}
   \end{theorem}

   \begin{proof}
   The assumptions of Theorem~\ref{N} are satisfied, hence, for a neighborhood $O$ of $\bx$, if $\{x_k\}$ has  a tail in $O$, then $x_k \to \bx$ as $k \to +\infty$. Using standard calculus, we find $r>0$ and $L > 0$ such that
$$
\|f(\bx) - f(u) - Df(u) (\bx - u) \| \leq L \|u - \bx\|^{1+\alpha} \quad \mbox{whenever} \quad u \in \ball_r(\bx).
$$
In view of {Theorem~\ref{thmFD}},  adjust $r$, if necessary, and choose a constant $\lambda > 0$ such that
$$
 \|x - \bx \|^{{\frac{1}{q}}} \leq \lambda d(0, f(\bx) + Df(u)(x - \bx) + F(x)) \quad \mbox{whenever} \quad u,x \in \ball_r(\bx).
$$
Let $N \subset \N$ be any infinite set for which
$ x_k \in \ball_r(\bx)$ for all $k \in N$. Fix  $k \in N$. Using the inclusion $$f(\bx) - f(x_k) + Df(x_k)(x_{k} - \bx) \in f(\bx) + Df(x_k)(x_{k+1} - \bx) + F(x_{x+1}),$$ we obtain
\begin{eqnarray*}
 \|x_{k+1} - \bx\|^{{\frac{1}{q}}} &\leq& \lambda d(0,f(\bx) + Df(x_k)(x_{k+1} - \bx) + F(x_{k+1})) \\
 & \leq & \lambda \|f(\bx) - f(x_k) + Df(x_k)(x_{k} - \bx)\| \leq L \lambda \|x_k - \bx\|^{1+\alpha}.
\end{eqnarray*}
This gives us the desired convergence rate.
\end{proof}

\subsection{ Inexact quasi-Newton method}

In this subsection we consider an inexact version of the Newton  method \eqref{New} for solving \eqref{geneq} of the form
\be \label{inN}f(x_k)+ B_k(x_{k+1}-x_k) + r_k(x_k) + F(x_{k+1})\ni 0, \ee
where $\{B_k\}$ is a sequence in ${\cal L}(X,Y)$ which represents an approximation of the derivative of $f$ provided by, for example, Broyden update, BFGS, and alike. The sequence of functions $r_k: X \to Y$ represents inexactness. The following theorem extends Theorem~\ref{N} to the iteration \eqref{inN} and can be regarded as a version of the Dennis-Mor\'e theorem for generalized
equations;
for related results see \cite{DM}:

\begin{theorem}\label{iN} Suppose that
the function $f$ is  Fr\'echet differentiable at a solution $\bx$ of \eqref{geneq} and the mapping $f+F$  is strongly subregular at $\bx$ for $0$.
Then
 there exists a neighborhood $O$ of $\bx$ such that if a  sequence $\{x_k\}$  is generated by
the  method \eqref{inN},
  has a tail in $O$
and also
\begin{equation} \label{eqlimDfBk}
\lim_{k\to +\infty}\frac{\|(Df(\bx) - B_k)(x_{k+1}-x_k)\|+ \|r_k(x_k)\|}{\|x_{k+1}-x_k\|} = 0,
\end{equation}
then $\{x_k\}$ is superlinearly convergent to $\bx$.
\end{theorem}

\begin{proof}
By the  definition of the  Fr\'echet differentiability of $f$ at $\bx$, for each $\mu > 0$ there is $\delta>0$ such that
\begin{equation} \label{fufvb}
 \|f(u) - f(\bx) - Df(\bx)(u-\bx)\| \leq \mu \|u - \bx\| \quad \mbox{for each} \quad u \in \ball_{\delta} (\bx).
\end{equation}
Corollary~\ref{prop1} implies that  $f+F$ is strongly subregular at $\bx$ for $0$ if and only if so is $H:=f(\bx) + Df(\bx)(\cdot - \bx)+F$, hence there are positive constants $\kappa$ and $a$ such that
\begin{equation} \label{xbxdb}
 \|x - \bx \| \leq \kappa d(0, H(x)) \quad \mbox{for each} \quad x \in \ball_a(\bx).
\end{equation}
Let $\delta > 0$ be such that \eqref{fufvb} holds with $\mu :=1/(4\kappa)$ and set $O=\ball_a(\bx)\cap\ball_\delta(\bx)$. Let $\{x_k\}$  be any sequence generated by \eqref{inN}
for which there is $k_0 \in \N$ such that  $x_k\in O$ for all $k \geq k_0$ and \eqref{eqlimDfBk} holds. Make $k_0$ bigger, if necessary, to have
$$
  \|(Df(\bx) - B_k)(x_{k+1}-x_k)\|+ \|r_k(x_k)\| \leq \frac{1}{4\kappa} \|x_{k+1}-x_k\| \quad \mbox{whenever} \quad k \geq k_0.
$$
   For any $k \ge k_0$ we have
$$
  f(\bx) - f(x_k) - Df(\bx)(\bx - x_k) + (Df(\bx)-B_k)(x_{k+1}- x_k) - r_k(x_k) \in  H(x_{k+1})
$$
and thus the combination
of   \eqref{fufvb} and \eqref{xbxdb}  implies that
\begin{eqnarray*}
 \|x_{k+1} - \bx \| &\leq& \kappa d(0, H(x_{k+1})) \\
 &\leq&  \kappa \|(Df(\bx) - B_k)(x_{k+1}-x_k)- r_k(x_k) - [f(x_{k}) - f(\bx) - Df(\bx)(x_{k} - \bx)] \| \\
  &\leq& \frac{1}{4} \|x_{k+1} - x_k\| + \frac{1}{4} \|x_{k} - \bx \|  \leq  \frac{1}{4} \|x_{k+1} - \bx \| + \frac{1}{2} \|x_{k} - \bx \|.
\end{eqnarray*}
Therefore $\|x_{k+1} - \bx \| \leq (2/3) \|x_{k} - \bx \|$ for each $k \ge k_0$. Hence $x_k \to \bx$ as $k \to +\infty$. To estimate the rate of convergence, let $\varepsilon > 0$ be arbitrary. Find $r > 0$ such that $\ball_r(\bx) \subset O$ and \eqref{fufvb} holds with $\mu:= \varepsilon/(\kappa(2+\varepsilon))$ and $\delta:=r$. Then there is $k_1 \in \N$ such that
$x_k \in \ball_r(\bx)$ and
$$
  \|(Df(\bx) - B_k)(x_{k+1}-x_k)\|+ \|r_k(x_k)\| \leq \frac{\varepsilon}{\kappa(2+\varepsilon)} \|x_{k+1}-x_k\| \quad \mbox{whenever} \quad k \geq k_1.
$$
As in preceding lines, for such an index $k$ we get
\begin{eqnarray*}
 \|x_{k+1} - \bx \| &\leq& \frac{\varepsilon}{2 + \varepsilon} \|x_{k+1} - x_k\| + \frac{\varepsilon}{2 + \varepsilon} \|x_{k} - \bx \|  \leq   \frac{\varepsilon}{2 + \varepsilon} \|x_{k+1} - \bx \| + \frac{2\varepsilon}{2 + \varepsilon} \|x_{k} - \bx \|.
\end{eqnarray*}
Therefore for any $k > k_1$ we have $  \|x_{k+1} - \bx \| \leq \varepsilon \|x_{k} - \bx \|$. Hence $x_k \to \bx$ superlinearly.
\end{proof}

In the same way, by mimicking Theorem~\ref{rem}
one can obtain a statement analogous to Theorem~\ref{iN} for a strongly $q$-subregular mapping, extending  a result in \cite{MO}.

 \subsection{Semismooth Newton method}

We continue our study of Newton method for solving  the generalized equation \eqref{geneq}
 where $f:\reals^n\to {\reals^m}$ is Lipschitz continuous but   not necessarily differentiable around a reference solution $\bx$. To  deal with
 a Newton-type iteration  we use the ``linearization" of $f+F$ at $\bx$
 of the form given by the mapping \eqref{HA}
where the matrix  $A$  is an arbitrarily chosen element of Clarke's generalized Jacobian. We consider  the following  version of  Newton's iteration: given $x_k$ choose $A_k \in {\partial_C} f(x_k)$ and then find $x_{k+1}$ which satisfies
\be\label{nsN}f(x_k) + A_k(x_{k+1}-x_k) + F(x_{k+1}) \ni 0. \ee
When the function $f$ in {\eqref{geneq}} is semismooth {(see the paragraph before Corollary~\ref{c2} for the definition)}, {this} method is usually referred  to as  the {\em semismooth}
Newton method.  Note that in the theorem below we assume that $f$ possesses the  {semismoothness} property  but do not use the  directional differentiability of $f$ which appears in its definition.

\begin{theorem}  Consider the
method \eqref{nsN} applied to  \eqref{geneq} with a solution $\bx$  for a function $f$ which is semismooth at $\bx$
and assume  that for each $A \in \partial_C f(\bx)$ the mapping $H_A$ defined in \eqref{HA} is strongly subregular at $\bx$ for $0$.
 Then there exists a neighborhood $O$
of $\bx$ such that  if a  sequence $\{x_k\}$  is generated by \eqref{nsN}
and  has a tail $\{x_k\}_{k\geq k_0}$
with $x_k\in O$ for all $k \geq k_0$, then $\{x_k\}$ is
  superlinearly convergent to $\bx$.
\end{theorem}

\begin{proof}
First we show that there are positive constants $\lambda$ and $a$ such that
  \begin{equation} \label{xbxd1}
 \|x - \bx \| \leq \lambda d(0, H_A(x)) \quad \mbox{whenever} \quad x \in \ball_a(\bx) \mbox{ and } A \in \partial_C f(\ball_a(\bx)).
\end{equation}
Since the set $\partial_C f(\bx)$ is compact, there exists  a constant $ \kappa > \sup_{A \in  \partial_C f(\bx) } \subreg(H_A; \bx \for 0)$ (cf. the proof of \eqref{bbbs}).
Fix any $\gamma \in (0, 1/(2\kappa))$.
The mapping  $\partial_C f$ is outer semicontinuous at $\bx$, hence there exists  $r > 0$ such that
$$
 \partial_C f (u) \subset \partial_C f (\bx) + \gamma \ball
 \quad \mbox{for each} \quad   u \in \ball_r(\bx).
$$
Compactness of the set $\partial_C f(\bx)$ implies that there is a finite set $\mathcal{A} \subset \partial_C f(\bar{x})$ such that
$ \partial_C f (\bx) \subset \mathcal{A} + \gamma \ball$.
Hence
\begin{equation} \label{hasbc}
  \partial_C f (\ball_r(\bx))  \subset \partial_C f (\bx) + \gamma \ball \subset  \mathcal{A} + 2\gamma \ball.
\end{equation}
Given $A \in \mathcal{A}$ there exists $\alpha_A \in (0,r)$ such that the mapping $H_A$ is strongly subregular at $\bx$ for $0$ with the constant $\kappa$ and neighborhood $\ball_{\alpha_A}(\bx)$. Let $  a:= \min_{A \in \mathcal{A}} \alpha_A $ and $ \lambda:=\kappa/(1 - 2\gamma\kappa)$.
Fix any $x$ and $A$ as in \eqref{xbxd1}. As $a < r$, using inclusion \eqref{hasbc} we find $\bar{A} \in \mathcal{A}$ with $\|A - \bar{A}\| \leq 2 \gamma$. Therefore
\begin{eqnarray*}
  \| x - \bx\| &\leq & \kappa d(0, {H}_{\bar{A}}(x)) =\kappa  d(-f(\bx) - \bar{A}(x-\bx), F(x))  \\
    &\leq  & \kappa d(-f(\bx) - A(x-\bx), F(x)) + \kappa \| (A - \bar{A})(x - \bx)\|  \\
     &\leq  & \kappa d(0, {H}_{A}(x)) + 2 \gamma\kappa   \|x - \bx\|.
\end{eqnarray*}
Since $2 \gamma \kappa < 1$ we get \eqref{xbxd1}.

{The semismoothness of $f$ implies that for each $\mu > 0$ there is $\delta>0$ such that
\begin{equation} \label{fufv1}
 \|f(u) - f(\bx) - A(u-\bx)\| \leq \mu \|u - \bx\| \quad \mbox{whenever} \quad u\in \ball_{\delta} (\bx) \mbox{ and } A \in \partial_Cf(u).
\end{equation}

Let $\delta > 0$ be such that \eqref{fufv1} holds with $\mu =1/(2\lambda)$ and set $O=\ball_a(\bx)\cap\ball_\delta(\bx)$. Let $\{x_k\}$  be any sequence generated by  \eqref{nsN}
such that  $x_k\in O$ for all $k \in \N$.
Fix any $k \in \N$. As $f(\bx) - f(x_k) + A_k(x_{k} - \bx) \in H_{A_k}(x_{k+1})$ and $A_k \in \partial_C f(x_k)$, using \eqref{xbxd1} and \eqref{fufv1}, we get
\begin{eqnarray*}
 \|x_{k+1} - \bx\| &\leq& \lambda d(0,H_{A_k}(x_{k+1}))
  \leq  \lambda \|f(\bx) - f(x_k) + A_k(x_{k} - \bx)\| \leq \frac{1}{2}  \|x_{k} - \bx \|.
\end{eqnarray*}
Hence $x_k \to \bx$ as $k \to +\infty$. To establish the rate of convergence, let $\varepsilon > 0$ be arbitrary. Find $r > 0$ such that $\ball_r(\bx) \subset O$ and \eqref{fufv1} holds with $\mu= \varepsilon/\lambda$ and $\delta=r$. Then there is $k_1 \in \N$ such that
$x_k \in \ball_r(\bx)$ whenever $k > k_1$. As above, for such an index $k$, we get
\begin{eqnarray*}
 \|x_{k+1} - \bx\| &\leq& \lambda d(0,H_{A_k}(x_{k+1}))
  \leq  \lambda \|f(\bx) - f(x_k) + A_k(x_{k} - \bx)\|  \leq \varepsilon  \|x_{k} - \bx \|.
\end{eqnarray*}
Hence $x_k \to \bx$ superlinearly.}
\end{proof}

\begin{remark} \rm
In view of Corollary~\ref{c1}, the assumptions of the above theorem imply that the mapping $f+F$ is strongly subregular at $\bx$ for $0$.
\end{remark}

If one  considers \eqref{inN}  instead of \eqref{nsN}, by using the above arguments
one can   obtain a slight generalization of \cite[Theorem~3.2 (ii)]{CDG}.

\subsection{ Strong subregularity of Newton sequences}

{Denote by $\ell_\infty$ the space of (infinite) sequences  $\{x_k\}$  in $X$ with elements $x_1$, $x_2$, $\dots$, $x_k$, $\dots$ equipped with the norm $\|\{x_k\}\|_\infty = \sup_{k \in \N}\|x_k\|.$
Consider the mapping
$$Y \times X \ni(p,u) \mapsto {\cal S}(p,u) = \left\{\{x_k\}  \mid f(x_k) + Df(x_k)(x_{k+1}- x_k) +F(x_{k+1}) \ni p, \  x_0 = u \right\},$$
that is, ${\cal S}(p,u) $ is the set of all sequences generated by the (perturbed) Newton method starting from the point $u$. Note that if $(x,p) \in \gph (f+ F)$, then the constant sequence $\{x\} \in {\cal S}(p,x)$.  In particular, if  $\bx$ is a solution of \eqref{geneq}, then $\{\bx\} \in {\cal S}(0, \bx)$.}

{\begin{theorem}
Suppose that $f$ is  Fr\'echet differentiable around
$\bx$
and $Df$ is continuous at $\bx$. The mapping $f+F$ is strongly subregular at $\bx$ for $0$ if and only if there is $\lambda > 0$ such that for any $\gamma \in (0,1)$ there is $a>0$ with the property that
for each $\{x_k\} \in \ball_a(\{\bx\})$ and each $(p,u) \in {\cal S}^{-1}(\{x_k\}) \cap (Y \times \ball_a(\bx))$ we have
\begin{equation} \label{subregmod}
  \|\{x_k\} - \{\bx\}\|_\infty \leq \gamma \|u - \bx\| + \lambda  \, \|p\|.
\end{equation}
In this case, the infimum of such constants $\lambda$ is equal to $\subreg (f+F;\bx\for 0)$.
\end{theorem}}

\begin{proof} {Denote by $i$ the infimum of $\lambda > 0$ such that for any $\gamma \in (0,1)$ there is $a>0$ such that inequality \eqref{subregmod} holds for each $\{x_k\} \in \ball_a(\{\bx\})$ and each $(p,u) \in {\cal S}^{-1}(\{x_k\}) \cap (Y \times \ball_a(\bx))$.}

{First, assume that $i < +\infty$ and fix any $\lambda > i$.  Pick any $\gamma \in (0,1)$. Then there is $a>0$ such that  for each $\{x_k\} \in \ball_a(\{\bx\})$ and each $(p,u) \in {\cal S}^{-1}(\{x_k\}) \cap (Y \times \ball_a(\bx))$ we have
\be\label{sec}
  \sup_{k\in \N}\|x_k-\bx\| \leq \gamma \|u-\bx\|+ \lambda \|p\|.
\ee
Let $x \in \ball_a(\bx)$ be arbitrary. Pick arbitrary  $p \in f(x) + F(x)$ (if  any). Then the constant sequence
$\{x\} \in {\cal S}(p,x)\cap \ball_a(\{\bx\})$, hence it satisfies \eqref{sec}, that is
$$
 \|x-\bx\| \leq \gamma \|x-\bx\|+ \lambda \|p\|,
$$
which yields
$$
\|x-\bx\| \leq \frac{\lambda}{1-\gamma} \|p\|.
$$
As $p \in f(x) + F(x)$ was arbitrary, we conclude that $f+F$ is strongly subregular at $\bx$ for $0$ with the constant $\lambda/(1-\gamma)$ and neighborhood $\ball_a(\bx)$. Letting $\gamma \downarrow 0$ we get that $\subreg (f+F;\bx\for 0)\leq \lambda$, and consequently $\subreg (f+F;\bx\for 0)\leq i$. }

{Assume that $f+F$ is strongly subregular at $\bx$ for $0$. Fix any $\lambda > \subreg (f+F;\bx\for 0)$ and any  $\gamma \in (0,1)$. Without  loss of generality assume that $\gamma$ is small enough to have that $\kappa:=\lambda(1-\gamma)/(1+\gamma)>\subreg (f+F;\bx\for 0)$.
Find $a > 0$  such that
\begin{equation} \label{smrfF}
 \|u - \bx \| \leq \kappa d(0, f(u) + F(u)) \quad \mbox{for each} \quad u \in \ball_a(\bx).
\end{equation}
Let $\mu := \gamma/(\kappa(1+\gamma))$.
Continuous differentiability of $f$ implies that, we can make $a $ smaller, if necessary, so that
\be\label{mub}
 \|f(u) - f(v) - Df(v)(u-v) \|\leq \mu\|u-v\| \quad \mbox{for all} \quad u,v \in \ball_a(\bx).
\ee
Fix any sequence  $\{x_k\} \in \ball_a(\{\bx\})$.  Pick arbitrary  $(p,u) \in \mathcal{S}^{-1} (\{x_k\})\cap(Y \times \ball_a(\bx))$ (if  any).
Note that
\be\label{xkb}
f(x_{k}) - f(x_{k-1}) -Df(x_{k-1})(x_{k}-x_{k-1}) +p  \in (f+F)(x_{k}) \quad \mbox{for each} \quad k\in \N.\ee
Fix any index $k \in  \N$, then  \eqref{smrfF}, \eqref{xkb}, and  \eqref{mub} imply that
\begin{eqnarray*}
 \|x_{k} - \bx \| &\leq& \kappa d(0,f(x_{k})+F(x_{k})) \leq \kappa \|f(x_{k}) - f(x_{k-1}) -Df(x_{k-1})(x_{k}-x_{k-1}) +  p \|\\
 & \leq &  \kappa ( \mu\|x_{k}-x_{k-1}\| +  \|p\|) \leq \kappa \mu \|x_{k} - \bx \| + \kappa \mu \|x_{k-1} - \bx\| + \kappa \|p\|.
\end{eqnarray*}
Noting that $\gamma(1-\kappa\mu) = \kappa\mu$ and $\kappa \mu (1+\gamma) = \gamma$, we get
\begin{equation} \label{estimxk0}
 \|x_{k} - \bx \|  \leq \gamma \|x_{k-1} - \bx\| + \kappa (1 + \gamma)\|p\| \quad \mbox{for each} \quad k \in \N.
\end{equation}
We claim that
\begin{equation} \label{estimxk}
  \|x_{k} - \bx \|  \leq  \gamma^{k} \|u - \bx\| + \kappa (1 + \gamma) \frac{1-\gamma^{k}}{1 - \gamma} \|p\|  \quad \mbox{for each} \quad k \in \N.
\end{equation}
Indeed, as $x_0 = u$, \eqref{estimxk0} with $k=1$ is  \eqref{estimxk} for $k=1$. We proceed by induction, assume that \eqref{estimxk} holds for some $k:=k_0 \in \N$. This and \eqref{estimxk0} with $k=k_0+1$ imply that
\begin{eqnarray*}
 \|x_{k_0+1} - \bx \|  &\leq & \gamma   \|x_{k_0} - \bx\| + \kappa (1 + \gamma)\|p\|
\leq \gamma^{k_0+1} \|u - \bx\| + \kappa (1 + \gamma) \|p\| \left(\frac{\gamma-\gamma^{k_0+1}}{1 - \gamma}  + 1 \right)\\
& =&  \gamma^{k_0+1} \|u - \bx\| + \kappa (1 + \gamma) \frac{1-\gamma^{k_0+1}}{1 - \gamma} \|p\|,
\end{eqnarray*}
which is \eqref{estimxk} for $k:=k_0+1$.  Inequality \eqref{estimxk} is proved.
Noting that $\gamma < 1$ we have
\begin{eqnarray*}
 \sup_{k \in \N} \|x_{k}-\bx\|  & \leq &  \gamma \| u - \bx\| + \frac{\kappa(1+\gamma)}{1-\gamma}\|p\| =  \gamma \| u - \bx\| + \lambda \|p\|.
\end{eqnarray*}
As  $(p,u) \in {\cal S}^{-1}(\{x_k\}) \cap (Y \times \ball_a(\bx))$ was arbitrary,  the mapping ${\cal S}^{-1}$  is strongly subregular at $\{\bx\}$ for $(0, \bx)$ and \eqref{subregmod} holds. Clearly, $i \leq \lambda$, hence $i \leq \subreg (f+F;\bx\for 0)$. }
\end{proof}

\section{Applications to optimization}\label{appl}

\subsection{Nonlinear programming}

In this subsection we study strong subregularity of a mapping which plays a major role in the nonlinear programming problem
\be\label{np} \text{minimize} g_0(x) \ee
subject to equality and inequality constraints:
\be\label{constr}
\left\{\begin{array}{ll}  g_i(x) = 0&   \text{for} i = 1,2, \dots, s,\\
                             g_i(x) \leq 0&   \text{for} i =s+1,\dots, m,\end{array}\right.
\ee
where the functions $g_i:\reals^n \to \reals$, $i=0,1,\dots, m$ are twice continuously differentiable everywhere.
Under a constraint qualification condition which will be specified a bit later, the first-order necessary optimality condition
is represented by the Karush-Kuhn-Tucker  (KKT) system
\be \label{KKT} \left\{\begin{array}{ll} \nabla_x L(x,y) &= 0,\\
\nabla_y L(x,y) &\in N_{\reals^s \times \reals_+^{m-s}}\end{array},\right.
\ee
where
$$L(x,y) = g_0(x) + \sum_{i=1}^m y_ig_i(x), \quad (x,y) \in \R^n\times \R^m,$$
is the Lagrangian associated with the problem \eqref{np}; here $y =(y_1,\ldots,y_m)$
is  the vector of Lagrange multipliers.
We study the strong subregularity of the following  mapping associated with the KKT system \eqref{KKT}:

\be\label{KKTmap}T: (x,y) \mapsto \left( \begin{array}{cc}  \nabla_x L(x,y)\\ -\nabla_yL(x,y) \end{array}\right) + N_{\reals^n\times\reals^s \times \reals_+^{m-s}}({x,y}).\ee

Let $(\bx,\by)$ be a reference solution of  \eqref{KKT}. Define the index sets
\bas I_1 & = & \{i\in \{s+1,\dots, m\} \mid g_i(\bx) =0, \by_i >0\} \cup \{1, \dots, s\},\\
I_2 & =& \{i\in \{s+1,\dots, m\} \mid g_i(\bx) =0, \by_i =0\},\\
I_3 & = & \{i\in \{s+1,\dots, m\}\mid g_i(\bx) <0, \by_i=0\}.\eas
In further lines we utilize the  following condition:
\be\label{smf}
\!\!\text{there is no nonzero} y\in \reals^m \text{such that} \sum_{i=1}^m y_i\nabla g_i(\bx) = 0 \text{and} y_i \geq 0, i \in I_2.
\ee
This condition implies the well-known  Mangasarian-Fromovitz Constraint Qualification (MFCQ) condition, in which the set $I_2$ is replaced by $I_1\cup I_2$. As well known, the MFCQ yields that the set of Lagrange multipliers for problem \eqref{np} satisfying \eqref{KKT} is nonempty, convex and compact. The  condition \eqref{smf} was introduced in \cite{Kyp} under the name Strict Mangasarian-Fromovitz Constraint Qualification. This name however does not reflect the nature of the condition since the latter is a condition on the optimality system while MFCQ is a condition on the constraint mapping; actually, MFCQ is equivalent to the metric regularity of that mapping.
Condition \eqref{smf}  implies
that the set of Lagrange multipliers consists of a single point; we will give a proof of this claim in the proof of the next theorem.

Denote $A= \nabla^2_{xx}L(\bx,\by)$ and $B=\nabla^2_{xy}L(\bx,\by)$;  that is, $B$ is the  $n\times m$ matrix whose rows are the vectors $\nabla g_i(\bx), i = 1,2,\dots, m.$
 Define the  so-called critical cone
$$K = \{x' \mid   \langle \nabla g_i(\bx), x' \rangle  = 0 \text{for} i \in I_1, \  \langle  \nabla g_i(\bx), x'  \rangle  \leq 0 \text{for} i \in I_2\}.$$
Recall that the second-order necessary condition for local optimality has the form
\be\label{sonc} \la x', Ax'\ra \geq 0 \text{for all }x' \in K,\ee
while the second-order sufficient condition is
\be\label{sosc}  \la x', Ax'\ra > 0 \text{for all }x'  \in K \setminus \{0\}.\ee
Now we are ready to state the main result of this subsection.

\begin{theorem}\label{nlp} The following are equivalent:
\paritem{(i)} The conditions \eqref{smf} and \eqref{sosc} are both satisfied;
\paritem{(ii)} The KKT mapping $T$ defined in \eqref{KKTmap} is strongly subregular at $(\bx,\by)$ for $0$ and $\bx$ is a
strong  local  minimizer of  \eqref{np}, meaning that there is a neighborhood $U$ of $\bx$ and a constant $\beta>0$ such that
    $${g_0}(x) \geq {g_0}(\bx) + \beta\|x-\bx\|^2 \text{for all} x \in U\cap C,$$
    {where $C:= \{ x\in \reals^n \mset  \ \eqref{constr} \mbox{ is satified } \}$.}
 \end{theorem}

\begin{proof} Linearizing the functions appearing in the mapping \eqref{KKTmap} at $(\bx, \by)$ we obtain
the mapping
\be\label{LKKT}L: (x,y) \mapsto \left(\begin{array}{ll} 0\\ {\bar g} \end{array}\right)+ \left( \begin{array}{cc}  A & B^T  \\ -B & 0 \end{array}\right)  \left( \begin{array}{cc}  x-\bx  \\ y-\by \end{array}\right) + N_{\reals^n\times\reals^s \times \reals_+^{m-s}}({x,y}),\ee
where we take into account that $\nabla_xL(\bx, \by) = 0$ and $g_i(\bx) = 0, i \in I_1\cup I_2$, and use the notation
$$\bar g = \left(\begin{array}{l} 0 \\ 0\\ -g_{I_3}(\bx)\end{array}\right),$$
in which $g_I$ is  a vector with components $g_i, i\in I$.
We can now apply Theorem~\ref{2} according to which the mapping $T$ in \eqref{KKTmap} is  strongly subregular at $(\bx,\by)$ for $0$ if and only if the mapping $L$ defined in \eqref{LKKT} has the same property. The graph of the   mapping $L$
is the union of polyhedral convex sets hence the strong subregularity of $T$ is equivalent to the property that
the  vector $(\bx, \by)$   is
an isolated point in $L^{-1}(0)$.

 Without loss of generality suppose that $I_1 = \{1,2,\dots, s_1\}$ and $I_2 = \{s_1 +1, \dots, s_2\}$. Denote by $B_1$ and $B_2$ the submatrices of $B$ corresponding to the index sets $I_1$ and $I_2$, respectively; that is, the rows of $B_1$ are the vectors $\nabla g_i (\bx), i = 1,2,\dots, s_1$, and analogously for $B_2$.

Let (i) hold.  We will now show that $(0,0)$ is the unique solution of the variational inequality
\bea
&& Ax +B^Ty=0, \label{a1} \\
&& B_1x = 0,    \label{a2}\\
&& B_2x \in N_{\reals^{I_2}_+}(y_{I_2}),  \label{a3}
\eea
where  $y_{I_2}$ is the subvector of $y$ whose components have indices in $I_2$ and  $\reals^{I_2}_+$ is the set of vectors $y_{I_2}$ with nonnegative components.
  Suppose that the mapping $T$ is not strongly subregular at $(\bx, \by)$ for $0$. Then there is a nonzero vector $(x,y)$ satisfying \eqref{a1}--\eqref{a3}. Assume that $x\neq 0$.
Multiplying \eqref{a1} by $x$ and taking into account \eqref{a2} and \eqref{a3} we obtain $\la x,Ax\ra = 0$ which contradicts \eqref{sosc}.  Hence $x=0$. But then there exists a nonzero $y\in \reals^m$ such that $B^Ty = 0$ and $0 \in N_{\reals_+^{I_2}}(y_{I_2})$,
hence $y_{I_2} \geq 0$. This contradicts \eqref{smf}.  Thus the mapping $T$
in \eqref{KKT} is  strongly subregular at $(\bx,\by)$ for $0$. It is a standard fact  that when $(\bx,\by)$ satisfies
\eqref{KKT} and  the second order sufficient condition \eqref{sosc}, then  $\bx$ is a strong  local solution of problem \eqref{np}. Hence,  (ii) is established.

In the opposite direction, suppose that the conditions in (ii) are satisfied. Then from the analysis in the beginning of the proof we conclude that
the  vector $(\bx, \by)$   as an isolated point in $L^{-1}(0)$. This in turn yields that $(0,0)$ is the unique solution of the variational inequality \eqref{a1}--\eqref{a3}. But this immediately implies \eqref{smf}. Furthermore, from
the assumed optimality of $\bx$  the second-order necessary condition \eqref{sonc} holds:
$$ \la x', Ax'\ra \geq 0 \text{for all nonzero}x' \in K.$$
We only need to show that this inequality is strict. On the contrary, suppose that there exists a nonzero $x' \in K$ such that
$Ax'=0$. Then the nonzero vector $(x',0)$ is a solution of \eqref{a1}--\eqref{a3}, a contradiction. Hence the conditions in (i) are satisfied.
\end{proof}

Theorem~\ref{nlp} partially extends   \cite[Theorem 2.6]{AT} with a new proof; in the latter theorem   it is also  shown that under the conditions in (i) there exist neighborhoods $U$ of $(\bx,\by)$ and $V$ of $0$ such that
for every $v \in V$ the   set  $ T^{-1}(v)\cap U$ is nonempty.

\subsection{A radius theorem}

A classical result, sometimes  called the Eckart-Young theorem, says that
 for any nonsingular matrix $A \in \reals^{n\times n}$,
$$\inf \lset \|B\| \mid A+B \text{singular} \rset
= \frac{1}{\|A^{-1}\|} \, . $$
A far reaching generalization of this result was proved in \cite{rad}, see also \cite[Theorem 6A.7]{book}, for the property of  metric regularity  of a set-valued mapping $F$ acting between Euclidean spaces. This result was extended later in \cite[Theorem 5.12]{regcon}, see also   \cite[Theorem 6A.9]{book}, to the property of strong subregularity as follows:

\begin{theorem} \label{11}
Consider a mapping
$F:\reals^n \tto \reals^m$  which  is  strongly subregular  at $\bx$ for $\by$.  Then
$$
 \inf _{B\in {\cal L}(\reals^n, \reals^m)}\! \Lset \displaystyle{\| B\| \Mset F+B
     \!\text{is not strongly   subregular at $\bx$  for  $\by +B\bx$} \Rset
       =\frac{1}{\subreg (F;\bx \for\by )}. }
      $$
Moreover, the infimum  remains unchanged  when either taken with respect to linear mappings
of rank 1 or enlarged to all  functions $f$ that are  calm at $\bx$, with
$\|B\|$ replaced by the calmness modulus $\,\clm (f;\bx)$ of $f$ at $\bx$.
\end{theorem}

Note that in  Theorem~\ref{11} the perturbation is represented by an {\em arbitrary} linear and bounded mapping $B$. In a number of cases, however, one should focus on mappings that have special structure. Such a situation arises in particular when one attempts to determine the ``radius of good behavior" of an optimization problem. To be specific, consider  the problem
\be\label{o}
         \text{minimize} \, g(x)  \text{  over} x \in C,
\ee
where $C$ is a nonempty  {\em polyhedral}  convex subset of $\reals^n$ and $g:\reals^n \to \reals$ is twice continuously differentiable everywhere.  The first-order necessary optimality condition for problem (\ref{o}) has the form
\be\label{vii}
          \nabla g(x) + N_C(x) \ni 0.
\ee
In the sequel the mapping $x \mapsto \nabla g(x) + N_C(x)$ is called the {\em optimality mapping.} Every solution of the variational inequality \eqref{vii} is said to be a {\em critical point}.
 The {\em critical cone} at $\bx$ for $-\nabla g(\bx)$ is defined as
$$
     K_C(\bx, -\nabla g(\bx)) = T_C(\bx)\cap [-\nabla g(\bx)]^\perp.
$$
 The  second-order
sufficient optimality condition for problem (\ref{o})  has the form
\be\label{sosc1}
   \langle u, \nabla^2 g(\bx) u \rangle > 0
                  \quad \text{for all nonzero}  u \in  K_C(\bx,-\nabla g(\bx)).
\ee
The following theorem is proved in \cite[Theorem 4G.4]{book}:

\begin{theorem}\label{thm1} Let $\bx $ be a critical point for (\ref{o}).
Then the
following are equivalent:
\paritem{(a)} the second-order sufficient condition \eqref{sosc1} holds at $\bx$;
\paritem{(b)}  the point $\bx$ is a local minimizer for problem \eqref{o} and
              the optimality  mapping $\nabla g +N_C$ is strongly subregular  at  $\bx$ for $0$.
  \newline\indent  In either case, $\bx$  is actually a strong local minimizer.
\end{theorem}

We now apply this last result to obtain  a  radius theorem for problem (\ref{o}). Let $\bx$ be a local minimizer for \eqref{o}. Along with \eqref{o} we consider the perturbed problem
\be\label{pp}\min\left[g(x)+ \frac{1}{2}\la x-\bx, B(x-\bx)\ra\right] \text{over} x\in C,\ee
where $B \in \reals^{n\times n}$ is  a symmetric matrix which enters the quadratic form representing the perturbation.

  \begin{theorem}\label{rado} Let $\bx$ be  a local minimizer for \eqref{o}, let
  $A = \nabla^2 g(\bx)$ and
  $K$ be the associated critical cone, and let
   the second-order sufficient condition (\ref{sosc1}) holds at $\bx$. Then
   \be\label{rr1}   \begin{array}{ll}\inf
\limits_{B \in \reals^{n\times n} {\rm symmetric}} \bigg\{\|B\|
  \mid&\text{$\bx$ is a local minimizer of \eqref{pp} and } \\
  &\text{the
second-order
condition for \eqref{pp}}\\
&\text{does not hold at $\bx$}\bigg\}\qquad
 = \min\limits_{\substack{{x \in K}\\{\|x\|=1}} }\la x, Ax\ra. \end{array}\ee
\end{theorem}

\begin{proof} From Theorem \ref{thm1} the quantity on the left side of \eqref{rr1} is the same as
the quantity
  \be\label{rr}   \begin{array}{ll}\inf

\limits_{B \in \reals^{n\times n} {\rm symmetric} } \bigg\{\|B\|
  \mid &\text{$\bx$ is a local minimizer of \eqref{pp} and  the optimality mapping } \\   &\text{ for  \eqref{pp} is  not strongly subregular
    at $\bx $ for $0$} \bigg\}. \end{array}\ee
Since the strong subregularity is stable under linearization, the optimality mapping  $x\mapsto \nabla g(x) +B(x-\bx)  +N_C(x) $ for \eqref{pp} is
not strongly subregular at $\bx $ for $0$ exactly when the mapping
$x\mapsto \nabla g(\bx) +(A+B)(x-\bx)  +N_C(x) $ is not strongly subregular at $\bx $ for $0$.
Then the  quantity  in (\ref{rr}) is the same as
$$   \begin{array}{ll}\inf
\limits_{B \in \reals^{n\times n} {\rm symmetric}  }\bigg\{\|B\|
  \mid &\text{$\bx$ is a local minimizer of \eqref{pp} and  the  mapping } \\   & \quad
\  x \mapsto \nabla g(\bx)  + (A+B)(x-\bx)  +N_C(x)  \\
  & \quad \text{is not strongly subregular
    at $\bx $ for $0$} \bigg\}.\end{array}$$
Since the critical cone $K$ remains the same for the perturbed problem \eqref{pp}, by Theorem~\ref{thm1}  the latter quantity equals
\be\label{l}i:=\inf_{B \in \reals^{n\times n} {\rm symmetric} }\{\|B\| \mid
A+B \text{not positive definite on $K$}\}. \ee
By assumption,   $A$ is symmetric positive definite on the cone $K$, thus we have
$$ \sigma:= \min_{\substack{{x \in K}\\{\|x\|= 1}} }\la Ax, x \ra  >0.$$
Let this minimum be attained for some $\tilde{x}$. The matrix
$$B =- \sigma \tilde{x}\tilde{x}^T$$
is symmetric (and negative definite). We have
$$\la \tilde{x}, (A+B) \tilde{x}\ra = \la \tilde{x},( A-\sigma\tilde{x}\tilde{x}^T) \tilde{x}\ra = \sigma - \sigma  = 0,$$
hence $A+B $ is not positive definite on $K$.
Moreover,
$$\|B\| =  \sup_{\|x\|= 1}  \|Bx\| = \sigma \| \tilde{x}\|  \sup_{\|x\|= 1} \tilde{x}^Tx = \sigma.$$
Thus
\be\label{i}  i \leq \sigma.\ee

To prove the opposite inequality, observe that
for any $n\times n$ matrix $B$ and any  $x \in K$, $\|x\|=1$, we have
$$\la x,(A+B)x \ra \geq \sigma  -  |\la x, Bx\ra |.  $$
Then
$$\min_{\substack{{x \in K}\\{ \|x\|= 1}}}\la x,(A+B)x \ra \geq \sigma -  \sup_{\substack{{x \in K}\\{ \|x\|=1}}}|\la x, Bx\ra |  \geq \sigma -  \sup_{\|x\|= 1 }|\la x, Bx\ra |>0 $$
provided that
$$\sup_{\|x\|=1} |\la x, Bx\ra | =\|B\| < \sigma. $$
Thus, for any symmetric $B$ such that $\|B\| < \sigma$, we have that $A+B$ is positive definite.
Hence,
$i \geq \sigma.$
Putting this together with \eqref{i} we obtain $i=\sigma$. This proves
that the quantity in (\ref{l})  equals the right side of (\ref{rr}).
\end{proof}

Note that when $C=\reals^n$ then the right side of (\ref{rr}) equals the smallest eigenvalue of $A$, which, as well known,
is equal to
the reciprocal of $\|A^{-1}\|$,  and we come to the finite-dimensional version of the extension of the Eckart-Young theorem described  in \cite{Z}: if $A$ is symmetric positive definite, then the
norm of the
smallest in norm symmetric matrix $B$
 such that $A+B$ is singular, equals $1/\|A^{-1}\|$. If $C$ is a subspace, then the radius quantity becomes  $1/\|(M^TAM)^{-1}\|$
where the columns of $M$ form a basis of  $C$.

Finally, we note that various versions of
Theorem~\ref{thm1} are available in the literature {as mentioned in the Introduction.} Theorem~\ref{rado} is new.

\subsection{Discrete approximations in optimal control}

Consider  the following optimal control problem with control constraints:
\be \label{Eoc}
     \text{minimize}  \ \int_0^1 \phi(y(t), u(t)) \dd t
\ee
subject to
$$  \dot{y}(t) =  g(y(t), u(t)), \quad y(0)=  0,\qquad  u(t) \in U \text{for a.e.} t \in  [0,1],$$
 where $\phi:\reals^{n+m} \to \reals$, $g:\reals^{n+m} \to \reals^n$,
$U$ is a  closed convex set in $\reals^m$ of feasible control values, $\dot{y}$ denotes the  derivative of the function $t\mapsto y(t)$ with respect to time $t$, and a.e. means almost every in the sense of Lebesgue measure.
 The admissible  controls $u $ are functions in
  $L^\infty([0,1],\reals^m)$, the space of essentially bounded and measurable functions on  $[0,1]$  with values in $\reals^m$, and
 the state trajectories $y$  belong to  $W^{1,\infty}_0([0,1],{\reals^n})$, the space of Lipschitz
continuous
functions   with  weak derivatives in $L^{\infty}([0,1],\reals^n)$ and value zero at $t=0$. In the sequel we sometimes  use the shortened notation $L^\infty(\reals^n)$ instead of $L^\infty([0,1],\reals^n	)$, etc. We assume that problem \eqref{Eoc} has a solution
$(\by, \bu)$
and also that there exists a closed set
$\Delta \subset \reals^n \times \reals^m$ and a $\delta > 0$
with $\ball_{\delta}(\by(t),\bu(t)) \subset \Delta$
for almost every $t \in [0,1]$ so that
the functions
$\phi$ and $g$ are twice continuously differentiable in an open set containing $\Delta$.

It is well known that { under some mild conditions which we will not reproduce here,}  the first-order necessary
condition { in normal form}  for a weak minimum, known under the name  the Pontryagin maximum principle,  at a solution $(\by, \bu)$  of problem (\ref{Eoc}) can be expressed in terms of
the Hamiltonian  $H(y, u, p) = \phi(y,u) + p^Tg(y,u)$  in the following way:
there exists   $\bp \in W^{1,\infty}(\reals^n)$, the so-called {\em adjoint variable},  such that
$\bx:=(\by,  \bu,  \bp)$ is a solution of the following two-point
boundary value problem coupled with a pointwise in $t$ variational inequality:
\be \label{OS} \left\{
\begin{array}{lll}
    \dot{y}(t)  & = &  g(y(t),u(t)), \quad  y(0) = 0 ,  \\
    \dot{p}(t) & = &  - {\nabla\!}_y H(y(t), u(t), p(t)), \quad p(1) = 0,  \\
     0  & \in &  {\nabla\!}_u H(y(t),u(t), p(t)) + N_U(u(t)),
     \end{array} \right.
\ee
for a.e. $t\in [0,1]$ where, as before,  $N_U(u)$ is the normal cone to the set $U$ at the point $u$.
Denote
$W_1^{1,\infty}(\reals^n) = \{ p \in W^{1,\infty}(\reals^n) \mid p(1) = 0\}$,
and let $X = W_0^{1,\infty}(\reals^n)\times  W_1^{1,\infty}(\reals^n)
\times  L^{\infty}(\reals^m)$ and
$Y= L^{\infty}(\reals^n)\times L^{\infty}(\reals^n)
\times L^{\infty}(\reals^m)$.
Further,  for $x = (y,u,p)$ let
\be \label{EOCf}
      f(x) = \left( \begin{array}{cc}
	               \dot{y} - g(y,u) \\
                   \dot{p}+{\nabla\!}_y H(y, u, p)\\
                   {\nabla\!}_u H(y ,u , p )\end{array}\right)
\quad \mbox{and} \quad
       F(x) =  \left( \begin{array}{cc} 0 \\ 0 \\ N_U(u)   \end{array}\right).
\ee	
The optimality system \eqref{OS} then takes the form of
 the generalized equation
 $0 \in f(x)+F(x),$ where $f:X\to Y$ and $F:X\tto Y$.
 In further lines we  will show  that strong subregularity of the  mapping
$f+F$ described by \eqref{EOCf}   for the optimality system \eqref{OS} provides a basis for obtaining an   error
estimate for a discrete approximation to this system.

Suppose that the optimality  system (\ref{OS}) is solved
inexactly by means of a numerical method applied to
a discrete approximation provided by the Euler scheme. Specifically,
let $N$ be a natural number, let $h = 1/N$
be the mesh spacing, and let $t_i = i h$, $i \in \{0,1, \dots, N\}$.
Denote by  $PL^N_0 (\reals^n)$ the space of piecewise linear and
continuous functions $y_N$ over the grid $\{t_i\}$ with values in $\reals^n$
and such that $y_N(0)=0$, by  $PL^N_1 (\reals^n)$ the space of
piecewise linear and continuous functions $p_N$ over the grid
$\{t_i\}$ with values in $\reals^n$
and such that $p_N(1)=0$, and by $PC^N (\reals^m)$ the space of piecewise constant and
continuous from the right functions over the grid $\{t_i\}$ with
values in $\reals^m$. Clearly, $PL^N_0 (\reals^n) \subset W_0^{1,\infty}(\reals^n)$, $PL^N_1 (\reals^n) \subset W_1^{1,\infty}(\reals^n)$ and  $PC^N (\reals^m) \subset L^{\infty}(\reals^m)$.
Then introduce the products
$X^N = PL^N_0 (\reals^n)\times PL^N_1 (\reals^n)\times PC^N(\reals^m)$
as an approximation space for the triple $(y, u,p)$. We identify $y \in PL^N_0 (\reals^n)$
with the vector $(y^0, \ldots, y^N)$ of its values at the mesh points, and similarly
for the adjoint variable $p$, and $u \in PC^N(\reals^m)$ is regarded as the vector $(u^0, \ldots, u^{N-1})$ of the
values of $u$ in the mesh subintervals.

Now, suppose that, as a result of the computations, for certain natural $N$
 a function $x_N =(y_N, u_N, p_N) \in X^N$ is found that satisfies the  {discrete optimality system}:
\be \label{DOS} \left\{  \begin{array}{lll}
	       y^{i+1} & = &y^i +  h
g(y^i, u^i),
\quad  \quad \quad \quad \quad \ \ \ y^0 = 0 ,  \\
           p^{i} & = &p^{i+1} + h{\nabla\!}_y H(y^{i}, u^i, p^{i+1}), \quad p^N = 0,  \\
           0  & \in &  {\nabla\!}_u H(y^i, u^i, p^i) + N_U(u^i)
           \end{array} \right.
\ee
for $i=0,1,\ldots, N-1$.  The system (\ref{DOS}) represents the Euler discretization
of the optimality system (\ref{OS}) {with  step-size $h=1/N$}.

Suppose that the mapping $f+F$, where $f$ and $F$ are  described in \eqref{EOCf}, is strongly subregular at $\bx$ for $0$.
Then there exist positive scalars $a $ and $\kappa$ such that if $x_N \in \ball_a(\bx)$, then
\bd
      \|x_N - \bx\|_X \leq \kappa d(0, f(x_N)+F(x_N)),
\ed
where the right side of this inequality is the  residual associated with the approximate solution $x_N$.
In our specific case, the residual can be estimated by the norm of a function $w_N \in Y$ defined
for each  $i \in \{0,1, \dots, N-1\}$ and $t \in [t_i,t_{i+1})$ as follows:
\bd
      w_N(t) = \left( \begin{array}{cc}
	            g(y_N(t_i), u_N(t_i))- g(y_N(t), u_N(t))  \\
   {\nabla\!}_y H(y_N(t_i), u_N(t_i), p_N(t_{i+1})) - {\nabla\!}_y H(y_N(t), u_N(t), p_N(t)) \\
 {\nabla\!}_u H(y_N(t_i), u_N(t_i), p_N(t_{i})) - {\nabla\!}_u H(y_N(t), u_N(t), p_N(t))\end{array}\right).
\ed
Thus, estimating the residual reduces to finding an estimate for the norm  $\|w_N\|_Y$.  By the definition of the norm in $Y$ we obtain
\bas
   \|w_N\|_Y & \leq & \max_{0\leq i \leq N-1} \sup_{t_i \leq t \leq t_{i+1}}
        \left[ \, \|g(y_N(t_i), u_N(t_i))- g(y_N(t), u_N(t))\| \right.\\
  & & + \|{\nabla\!}_y H(y_N(t_i), u_N(t_i), p_N(t_{i+1})) - {\nabla\!}_y H(y_N(t), u_N(t), p_N(t))\|\\
  & & \left.+  \|{\nabla\!}_u H(y_N(t_i), u_N(t_i), p_N(t_{i})) -
       {\nabla\!}_u H(p_N(t), u_N(t), p_N(t))\|\right].
\eas
Observe that here $y_N$ is a piecewise linear function across the grid $\{t_i\}$ with uniformly
bounded derivative, since both $y_N$ and $u_N$ are in some $L_\infty$ neighborhood
of $\by$ and $\bu$ respectively. Hence, taking into account that
the functions $g $, ${\nabla\!}_y H$, and ${\nabla\!}_u H$ are continuously
differentiable, this leads us to an estimate of order $O(1/N)$ for the error of the discretization.
Specifically, we obtain the following result:

\begin{theorem} \label{DOC} Assume that {the optimality mapping $f + F$ associated with (\ref{OS}), where $f$ and $F$ are defined in \eqref{EOCf},}
is strongly subregular at $\bx=(\by, \bu, \bp)$ for $0$. Then there exist {$N_0 \in \N$} and positive reals  $a$ and $c$  such that if for an integer $N \geq N_0$ a solution $x_N=(y_N,u_N,p_N)$  of the discrete optimality system \eqref{DOS} satisfies $\|x_N-\bx\|_X \leq a$
then
\be\label{da}
  \|x_N - \bx\|_X \leq \frac{c}{N}.
\ee
\end{theorem}

We should note that the assumption of strong subregularity of  the mapping associated with   (\ref{OS}) and
 considered as a mapping from $X = W_0^{1,\infty}\times  W_1^{1,\infty}
\times  L^{\infty}$ to
$Y= L^{\infty}\times L^{\infty}
\times L^{\infty}$ is quite strong. For example,
it follows from the estimate \eqref{da}
that if the reference optimal control  $\bu$ has a point of discontinuity in $t$, its piecewise constant discrete approximation
$u_N$ must have a jump at the same point. In the
paper \cite{DH},
see also  \cite{DMal},  strong regularity in $L^\infty$ is obtained under {\em coercivity} of the objective function, an assumption which automatically implies continuity of the optimal control $\bu$ as a function of time $t$. Without coercivity, for example,
when the problem is linear in control, one needs metric regularity in larger spaces, for some new results in this direction
see the recent paper \cite{QV}. In such spaces however, it may be not possible to differentiate, and hence to pass to a linearization.
 Theorem~\ref{DOC} should be treated as  a first step towards employing strong subregularity to obtain error estimates for discrete approximations in optimal control.

\end{document}